\documentclass[reqno,10pt, centertags]{amsart}
\usepackage{amsmath,amsthm,amscd,amssymb,latexsym,upref}
\usepackage{graphicx,tikz}
\usepackage{float}
\usepackage{latexsym}
\usepackage{lscape}
\usepackage{url}

\newcommand\numberthis{\addtocounter{equation}{1}\tag{\theequation}}

\newcommand{\beq}{\begin{equation}}
\newcommand{\enq}{\end{equation}}

\makeatletter
\def\theequation{\@arabic\c@equation}

\newcommand{\bbZ}{{\mathbb{Z}}}

\newcommand{\cQ}{{\mathcal Q}}

\newcommand{\bq}{{\mathbf q}}
\newcommand{\bp}{{\mathbf p}}

\newcommand{\bi}{\bibitem}

\allowdisplaybreaks
\numberwithin{equation}{section}

\renewcommand{\div}{\operatorname{div}}
\renewcommand{\det}{\operatorname{det}}

\newcommand{\curl}{\operatorname{curl}}

\newcommand{\diag}{\operatorname{diag}}

\theoremstyle{plain}
\newtheorem{theorem}{Theorem}[section]

\newtheorem{lemma}[theorem]{Lemma}

\theoremstyle{definition}

\newtheorem{property}[theorem]{Property}

\newtheorem{remark}[theorem]{Remark}

\newtheorem{notation}[theorem]{Notation}

\begin{document}
\allowdisplaybreaks

\title[Instability of unidirectional flows for 2D Navier-Stokes]{Instability of unidirectional flows for the 2D Navier-Stokes equations and related $\alpha$-models}
\thanks{The author is indebted to Professor Yuri Latushkin (University of Missouri-Columbia / Courant Institute of Mathematical Sciences) for several helpful discussions during the course of the preparation of this manuscript. The author also thanks Aleksei Seletskiy for making available the code from his forthcoming website and Yadugiri V. Tiruvaimozhi for help with the numerical illustrations.}

\author[S. Vasudevan]{Shibi Vasudevan}
\address{International Centre for Theoretical Sciences,
Tata Institute of Fundamental Research, Bengaluru, 560089, India}
\email{shibi.vasudevan@icts.res.in}

\date{\today}



\begin{abstract}
We study instability of unidirectional flows for the linearized 2D Navier-Stokes equations on the torus. Unidirectional flows are steady states whose vorticity is given by Fourier modes corresponding to a single vector $\mathbf p \in \mathbb Z^{2}$. Using Fourier series and a geometric decomposition allows us to decompose the linearized operator $L_{B}$ acting on the space $\ell^{2}(\mathbb Z^{2})$ about this steady state as a direct sum of linear operators $L_{B,\mathbf q}$ acting on $\ell^{2}(\mathbb Z)$ parametrized by some vectors $\mathbf q\in\mathbb Z^2$. Using the method of continued fractions we prove that the linearized operator $L_{B,\mathbf q}$ about this steady state has an eigenvalue with positive real part thereby implying exponential instability of the linearized equations about this steady state. We further obtain a characterization of unstable eigenvalues of $L_{B,\mathbf q}$ in terms of the zeros of a perturbation determinant (Fredholm determinant) associated with a trace class operator $K_{\lambda}$. We also extend our main instability result to cover regularized variants (involving a parameter $\alpha>0$) of the Navier-Stokes equations, namely the second grade fluid model, the Navier-Stokes-$\alpha$ and the Navier-Stokes-Voigt models.  
\end{abstract}

\maketitle
\normalsize

\section{Introduction}
Consider the incompressible Navier Stokes equations in vorticity form on the torus $\mathbb T^{2}$
\begin{equation}\label{ns}
\partial_{t} \omega + \mathbf u \cdot \nabla \omega = \nu \Delta \omega + \nu f,
\end{equation}
where the velocity $\mathbf u$ and the vorticity $\omega$ are related via the equation $\omega=\curl \mathbf u$ and incompressibility imposes the condition $\nabla \cdot \mathbf u =0$, $\nu > 0$ is a constant and $f$ is a forcing function. Putting $\nu=0$ formally gives us the Euler equations. The condition $\nabla \cdot \mathbf u =0$ guarantees the existence of a stream function $\psi$ which is related to the velocity via $\mathbf u = \nabla^{\perp}\psi= (\psi_{y}, -\psi_{x})$ and $\omega= \curl \mathbf u = -\Delta \psi$.
Fix a non-zero vector $\mathbf p \in \mathbb Z^{2} \backslash \{0\}$ in the integer lattice $\mathbb Z^{2}$ and consider a steady state solution to the 2D Navier Stokes equations \eqref{ns} on the torus $\mathbb T^{2}$ of the form
\begin{align}\label{uni}
\omega^{0}(\mathbf x) &= \frac{\Gamma}{2}e^{i \mathbf p \cdot \mathbf x} + \frac{\Gamma}{2}e^{-i \mathbf p \cdot \mathbf x} = \Gamma \cos (\mathbf p \cdot \mathbf x),\nonumber \\
\mathbf u^{0}(\mathbf x)&= (\psi^{0}_{y},-\psi^{0}_{x}) =\frac{\Gamma}{\|\mathbf p\|^{2}}\sin (\mathbf p \cdot \mathbf x) \mathbf p^{\perp}, \nonumber \\
f&= -\Delta \omega^{0}= \|\mathbf p\|^{2}\Gamma \cos (\mathbf p \cdot \mathbf x).
\end{align}
where $\mathbf p = (p_{1},p_{2})$, $\mathbf p^{\perp}=(-p_{2},p_{1})$, $\Gamma \in \mathbb R$ and $\psi^{0}(\mathbf x)= \frac{\Gamma}{\|\mathbf p\|^{2}}\cos (\mathbf p \cdot \mathbf x)$.  
Such a steady state with one non-zero Fourier mode (characterized by the vector $\mathbf p$) for the vorticity is called a \textit{unidirectional flow}. Two important special cases of the above are the following. Letting $\mathbf p =(0,1)$ we obtain a steady state whose vorticity is given by $\omega^{0}(x,y)=\cos y$. This steady state is sometimes referred to as the Kolmogorov flow in the literature. L.\ Meshalkin and Y.\ Sinai, in \cite{MS} proved linear instability of these states on the domain $\mathbb R \times \mathbb T$ (a periodic channel, i.e., infinitely long in $x$-direction and periodic boundary conditions in $y$) for the Navier Stokes equations using continued fractions, see also \cite{Y65}. The case when $\mathbf p=(0,m)$ was considered by S.\ Friedlander, W.\ Strauss and M.\ Vishik \cite{FSV97}, who proved linear instability of these states to the Euler equations on $\mathbb T^{2}$ extending the continued fractions methods of Meshalkin and Sinai, see also \cite{BFY99} by L.\ Belenkaya, S.\ Friedlander and V.\ Yudovich and \cite{FH98} by S.\ Friedlander and L.\ Howard. When $\mathbf p =(1,0)$ these states are called bar states in \cite{BW}. In a recent preprint, see \cite{CEW20}, the authors exhibit an interesting family of non-trivial steady states that are arbitrarily close to the Kolmogorov flow for the Euler equations, with consequences for the Navier-Stokes equations. 

Studying the spectrum of the linearized differential operator obtained by linearizing the Navier-Stokes or Euler equations about steady states such as those considered above is an important first step towards an understanding of the stability and dynamics of solutions to the full nonlinear equations. Besides the earlier cited works, see also \cite{Fr91, FZ98, L92, L92a} for related work on the use of continued fractions to study instability of steady states to the Navier-Stokes equations.

Unidirectional flows can be classified as being of type $0,I,II$ (depending on the vector $\mathbf p$) as explained in Remark \ref{classf} below. Flows of type $0$ are known to be stable for the Euler equations. Recently, H.\ Dullin, Y.\ Latushkin, R.\ Marangell, J.\ Worthington and the present author proved in \cite{DLMVW20} that unidirectional flows that are of type $I$ are linearly exponentially unstable for the 2D Euler and the 2D $\alpha$-Euler equations. A result of this type was unavailable for the Navier-Stokes equations and this served as a first inspiration for the current paper. We extend the instability results in \cite{DLMVW20} to cover the 2D Navier-Stokes equations, see Theorem \ref{mainthm} and Theorem \ref{mainthmipim} in Section 2 below. We mention \cite{L,LLS} as important precursors to \cite{DLMVW20}. See also \cite{WDM1,WDM2} for related work on stability results for the Euler equations.

A second aim of this paper is to obtain a characterization of unstable eigenvalues of the linearized Navier-Stokes vorticity operator in terms of roots of a perturbation determinant (Fredholm determinant) $\mathcal D(\lambda)$ associated with a trace class operator $K_{\lambda}$, see equation \eqref{mathcald} and Theorem \ref{fredchar} in Section 3 below. Such characterizations are widely used in the spectral theory of Schr\"{o}dinger operators, see for example \cite{GM04, GLM07} and the bibliography therein, but seem much less explored in the fluid dynamics literature. We attempt to fill this gap. 

This paper also stems from an extensive literature of studying various regularizations of the Navier-Stokes and Euler equations: the so called $\alpha$-models. The $\alpha$-Euler and Navier-Stokes-$\alpha$ models are, respectively, a regularization of the Euler and Navier-Stokes equations involving a regularization parameter $\alpha>0$. These (together with various related $\alpha$-) models were introduced and studied by C.\ Foias, D.\ Holm, J.\ Marsden, T.\ Ratiu and E.\ Titi in \cite{FHT01, HMR98, HMR98a} and find widespread use in various applications such as turbulence modeling, data assimilation etc., see for example \cite{ANT16, CFHOTW98}. In view of their widespread use in applications, the study of the stability properties of these $\alpha$-models assumes importance. For instance, the Navier-Stokes-$\alpha$ in vorticity form is represented by the following equations (on the two-torus $\mathbb T^{2}$)
\begin{equation*}
\partial_{t}\omega+\mathbf u_{f}\cdot \nabla \omega=\nu \Delta \omega+\nu f,
\end{equation*}
where $\alpha>0$, the filtered velocity $\mathbf u_{f}$ is related to the actual velocity $\mathbf u$ via the formula $\mathbf u=(I-\alpha^{2}\Delta)\mathbf u_{f}$ and the vorticity $\omega=\curl \mathbf u =\curl (I-\alpha^{2}\Delta)\mathbf u_{f}$. 
  
A third aim of the present work is to extend the instability results for the following $\alpha$-regularized models of the 2D Navier-Stokes equations: 
the second grade fluid model, the Navier-Stokes-$\alpha$ model and the Navier-Stokes-Voigt model, see Theorem \ref{mainthm-sg-nsa-nsv} in Section 4. See the introduction to Section 4 for equations of motion and a basic description of (the viscous version) of these regularized models (setting $\nu=0$ therein allows us to recover the inviscid version of these $\alpha$-models).  

Our manuscript is organized as follows. The rest of Section 1 is devoted to setting up function spaces and operators. It is convenient to consider instead of equation \eqref{ns} above, an equivalent equation involving Fourier coefficients $\omega_{\mathbf k}$ of $\omega$, see equation \eqref{nsfourier} below. A geometric decomposition allows us to decompose our linearized operator $L_{B}$ (see \eqref{Lb} below) acting on the space $\ell^{2}(\mathbb Z^{2})$ as a direct sum of linear operators $L_{B,\mathbf q}$ acting on $\ell^{2}(\mathbb Z)$ parametrized by some vectors $\mathbf q\in\mathbb Z^2$. In Section 2, we use continued fractions techniques to prove linear instability of unidirectional flows that are of type $I$, see Theorem \ref{mainthm} and Theorem \ref{mainthmipim}. Our main theorem \ref{mainthm} characterizes the existence of an unstable eigenvalue to the linearized Navier-Stokes operator in terms of roots of equations involving continued fractions (see equation \eqref{fg} and Remark \ref{root}). In Section 3, we characterize the unstable eigenvalues of the operator $L_{B,\mathbf q}$ in terms of the roots of a perturbation determinant (Fredholm determinant) $\mathcal D(\lambda)$ associated with a trace class operator $K_{\lambda}$, see equation \eqref{mathcald} and Theorem \ref{fredchar}. In case $\mathbf q$ is of type $I$, Theorem \ref{fredchar} allows us to establish a one to one correspondence between the roots of the continued fractions equation \eqref{fg} and the roots of the perturbation determinant $\mathcal D(\lambda)$, see Theorem \ref{1-1}. 
In Section 4, we extend our main instability theorems of Section 2 to cover the second grade fluid model, the Navier-Stokes-$\alpha$ and the Navier-Stokes-Voigt models. Our main result in this Section is Theorem \ref{mainthm-sg-nsa-nsv}.

\subsection*{Problem setup and governing equations}
One can use Fourier series decomposition  $\omega(\mathbf x)=\sum_{\mathbf k\in\mathbb Z^2\setminus\{0\}}\omega_\mathbf k e^{i \mathbf k\cdot\mathbf x}$  
and rewrite \eqref{ns} as
\begin{equation}\label{nsfourier}
\frac{d\omega_{\mathbf k}}{dt}=\sum_{\mathbf q\in\mathbb Z^2\setminus\{0\}}\beta(\mathbf k-\mathbf q,\mathbf q)\omega_{\mathbf k-\mathbf q}\omega_\mathbf q - \nu \|\mathbf k\|^{2}\omega_{\mathbf k} + \nu f_{\mathbf k} ,\,\,\mathbf k\in\mathbb Z^2\setminus\{0\},
\end{equation}
where the coefficients $\beta(\mathbf p,\mathbf q)$ for $\mathbf p,\mathbf q\in\mathbb Z^2$ are defined as 
\begin{equation}\label{dfn}
\beta(\mathbf p,\mathbf q)=\frac{1}{2}\bigg(\|\mathbf q\|^{-2}-\|\mathbf p\|^{-2}\bigg)(\mathbf p\wedge\mathbf q)\,
\end{equation}
for %
$\mathbf p\neq0,\mathbf q\neq0$, and $\beta(\mathbf p,\mathbf q)=0$ otherwise.\footnote{A derivation of the equation analogous to \eqref{nsfourier} for the Euler equation ($\nu=0$) is given in Lemma 6.1 in the Appendix of \cite{DLMVW20}. The arguments therein can be adapted to derive \eqref{nsfourier} from \eqref{ns}.} Here 
\begin{equation}\label{dfnwedge1}
\mathbf p\wedge\mathbf q=\det\left[\begin{smallmatrix}p_1&q_1\\ p_2&q_2\end{smallmatrix}\right] \mbox{ for } \mathbf p=(p_1,p_2) \mbox{ and } \mathbf q=(q_1,q_2).
\end{equation}
The steady state \eqref{uni} satisfies the relation $\omega^{0}= \|\mathbf p\|^{2} \psi^{0}$.
Consider the linearization of \eqref{ns} about this steady state given by
\begin{equation}\label{nslin}
\partial_{t}\omega + \mathbf u^{0} \cdot \nabla \omega + \mathbf u \cdot \nabla \omega^{0}= \nu \Delta \omega.
\end{equation} 
Corresponding to \eqref{nslin}, consider the linear operator $L_{B}$ given by
\begin{equation}\label{Lb}
L_{B} \omega = - \mathbf u^{0} \cdot \nabla \omega - \mathbf u \cdot \nabla \omega^{0}+ \nu \Delta \omega
\end{equation}
Using the Fourier decomposition in \eqref{nsfourier}, consider the following linearized vorticity operator in the space $\ell^2(\mathbb Z^2)$
\begin{align}
L_B&: (\omega_\mathbf k)_{\mathbf k\in\mathbb Z^2}\mapsto
\big(\beta(\mathbf p,\mathbf k-\mathbf p)\Gamma\omega_{\mathbf k-\mathbf p}-
\beta(\mathbf p,\mathbf k+\mathbf p)\Gamma \omega_{\mathbf k+\mathbf p}- \nu \|\mathbf k\|^{2}\omega_{\mathbf k}\big)_{\mathbf k\in\mathbb Z^2}.\label{dfnLB}
\end{align}

\subsubsection*{Decomposition of subspaces and operators}\label{B-S} Similar to the Euler case considered in \cite[pp. 2054-2057]{DLMVW20}, we decompose the operator $L_B$ acting in $\ell^2(\mathbb Z^2)$ into the direct sum of operators $L_{B,\mathbf q}$, $\mathbf q\in\mathcal Q\subset\mathbb Z^2$, acting in the space $\ell^2(\mathbb Z)$, for some set $\mathcal Q\subset\mathbb Z^2$. 
Fix $\mathbf p\in\mathbb Z^2 \backslash \{0\}$ and for any $\mathbf q\in\mathbb Z^2$ we denote $\Sigma_{B,\mathbf q}=\{\mathbf q+n\mathbf p: n\in\mathbb Z\}$. Let $\widehat{\mathbf q}$ be the point in $\Sigma_{B,\mathbf q}$ with the smallest norm (there may sometimes be two such points in which case we let $\widehat{\mathbf q}$ be the point $\mathbf q+n\mathbf p \in \Sigma_{B,\mathbf q}$ with the larger $n$). For example, see Figure \ref{Classes of points} below, if $\mathbf q=(2,3)$ on the green line, $\Sigma_{B,\mathbf q}$ is all the dotted points on the green line and $\widehat{\mathbf q}$ will be the point $(-1,2)$. Let $\mathcal Q=\{\widehat{\mathbf q}(\mathbf q): \mathbf q\in\mathbb Z^2\}$. For each $\mathbf q\in\mathcal Q$ we denote by $X_{B,\mathbf q}$ the subspace of $\ell^2(\mathbb Z^2)$ of sequences supported in $\Sigma_{B,\mathbf q}$, i.e., $X_{B,\mathbf q}=\{(\omega_\mathbf k)_{\mathbf k\in\mathbb Z^2}: \omega_\mathbf k=0 \text{ for all $\mathbf k\notin\Sigma_{B,\mathbf q}\}$}$. Clearly, $\ell^2(\mathbb Z^2)=\oplus_{\mathbf q\in\mathcal Q}X_{B,\mathbf q}$, the operator $L_B$ leaves $X_{B,\mathbf q}$ invariant,  and therefore $L_B=\oplus_{\mathbf q\in\mathcal Q}L_{B,\mathbf q}$ where $L_{B,\mathbf q}$ is the restriction of $L_B$ onto $X_{B,\mathbf q}$. Each $X_{B,\mathbf q}$ is isometrically isomorphic to $\ell^2(\mathbb Z)$ via the map $(\omega_{\mathbf q+n\mathbf p}) \mapsto (w_n)_{n\in\mathbb Z}$.
Under this isomorphism the operator $L_{B,\mathbf q}$ in $X_{B,\mathbf q}$ induces an operator in $\ell^2(\mathbb Z)$ (that we will still denote by $L_{B,\mathbf q}$) given by the formula
\begin{align}
L_{B,\mathbf q}: (w_n)_{n\in\mathbb Z}\mapsto
\bigg( \beta(\mathbf p,\mathbf q+(n-1)\mathbf p)\Gamma w_{n-1}-&
\beta(\mathbf p,\mathbf q+(n+1)\mathbf p)\Gamma w_{n+1} \nonumber \\&- \nu \|\mathbf q + n \mathbf p \|^{2} w_{n} \bigg)_{n\in\mathbb Z}.\label{LBq}
\end{align}
If $\mathbf q$ is parallel to $\mathbf p$, then $L_{B,\mathbf q}$ above is the zero operator (since $\mathbf p \wedge \mathbf q = 0$ and $\mathbf p \wedge (\mathbf q + n \mathbf p)=\mathbf p \wedge \mathbf q$). Thus, in what follows we assume $\mathbf q $ is not parallel to $\mathbf p$.
Introduce the notation 
\begin{align}\label{dfnrho}
\rho_n=\Gamma\beta(\mathbf p,\mathbf q+n\mathbf p)=\frac{1}{2}\Gamma(\mathbf q\wedge\mathbf p)  \bigg(\frac{1}{\|\mathbf p\|^{2}} -\frac{1}{\|\mathbf q+n\mathbf p\|^{2}}\bigg), \, n\in\mathbb Z,
\end{align}
Assuming that $\Gamma$ satisfies the normalization condition $\frac{\Gamma(\mathbf q\wedge\mathbf p)}{2\|\mathbf p\|^{2}}=1,$ will imply that
$\rho_{n}$ satisfies
\begin{equation}\label{rhon}
\rho_{n}=1-\frac{\|\mathbf p\|^{2}}{\|\mathbf q+n\mathbf p\|^{2}}.
\end{equation}
Using the notation for $\rho_{n}$, the operator $L_{B,\mathbf q}$ in \eqref{LBq} can be rewritten as 
\begin{equation}\label{Lbqnew}
L_{B,\mathbf q}: (w_n)_{n\in\mathbb Z}\mapsto (\rho_{n-1}w_{n-1}-\rho_{n+1}w_{n+1}- \nu \|\mathbf q + n \mathbf p \|^{2} w_{n})_{n\in\mathbb Z}.
\end{equation}

\begin{figure}\label{Classes of points}
  \begin{center}
    \begin{tikzpicture}[scale=.8]
    \draw[help lines, thick] (-4,-4) grid (5,4);
    \draw [->] [thick] (-4.3,0) -- (5.3,0);
     \draw [->] [thick] (0,-4.3) -- (0,4.3);
    \node at (-.2,-.2) {$0$};
     \draw [->] [ultra thick] (0,0) -- (3,1);
     \node at (3.2,1.2) {$\bf p$};
     \draw [ultra thick, green] (-4,1) -- (5,4);
      \node at (-.4, 2.3) {${\bf q}_1$ (type $I_0$)};
      \node at (2.2, 3.2) {${\bf q}_1+{\bf p}$};
       \node at (-4, 1.3) {${\bf q}_1-{\bf p}$};
      \draw [fill] (-1,2) circle [radius=.1];
     \draw [fill] (2,3) circle [radius=.1];
      \draw [fill] (-4,1) circle [radius=.1];
       \draw [fill] (5,4) circle [radius=.1];
      \draw [ultra thick, blue] (-4,0) -- (5,3);
       \draw [fill] (-1,1) circle [radius=.1];
     \draw [fill] (2,2) circle [radius=.1];
      \draw [fill] (-4,0) circle [radius=.1];
          \draw [fill] (5,3) circle [radius=.1];
       \node at (-.5, .6) {${\bf q}_2$ (type $II$)};
        \node at (2, 1.5) {${\bf q}_2+{\bf p}$};
         \node at (5.2, 2.5) {${\bf q}_2+2{\bf p}$};
       \draw [ultra thick, red] (-4,-3.3333) -- (5,-.3333);
        \draw [fill] (0,-2) circle [radius=.1];
        \draw [fill] (3,-1) circle [radius=.1];
         \node at (0, -1.5) {${\bf q}_3$ (type $I_+$)};
     \draw [ultra thick, brown] (-4,-4) -- (5,-1);
      \node at (3.3, -2.5) {${\bf q}_4$ (type $I_-$)};
      \node at (-1, -3.5) {${\bf q}_4-{\bf p}$};
       \draw [fill] (-1,-3) circle [radius=.1];
        \draw [fill] (2,-2) circle [radius=.1];
         \draw [fill] (5,-1) circle [radius=.1];
          \draw [fill] (-4,-4) circle [radius=.1];
        \draw [fill] (-3,-3) circle [radius=.1];
     \draw[ultra thick] (0,0) circle [radius=3.1622];
\end{tikzpicture}
\caption{${\bf p}=(3,1)$; point ${\bf q}_1=(-1,2)$ is a point of type $I_0$ (green $\Sigma_{{\bf q}_1}$), point ${\bf q}_2=(-1,1)$ is a point of type $II$ (blue $\Sigma_{{\bf q}_2}$), point ${\bf q}_3=(0,-2)$ is a point of type $I_+$ (red $\Sigma_{{\bf q}_3}$), and point ${\bf q}_4=(2,-2)$ is a point of type $I_-$ (brown $\Sigma_{{\bf q}_4}$ ).}
\end{center}
\end{figure}
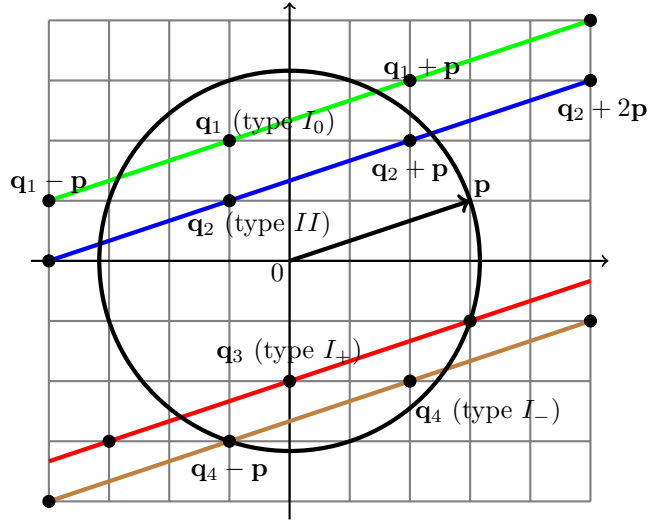
\begin{remark}\label{classf}
Points $\mathbf q \in \mathbb Z^{2}$ are said to be of type $0, I, II$ respectively, if the set $\Sigma_{B,\bq}=\{\bq+n\bp:n \in \bbZ\}$ contains zero, one or two points in the open disk of radius $\mathbf p$. Points of type $I$ are further classified as follows: $\mathbf q$ is of type $I_{0}$ if $\hat{\mathbf q}$ is inside the open disk of radius $\mathbf p$ and both $\hat{\mathbf q}+\mathbf p$ and $\hat{\mathbf q}-\mathbf p$ have norms greater than $\mathbf p$, $\mathbf q$ is of type $I_{+}$ if $\hat{\mathbf q}$ is inside the open disk of radius $\mathbf p$ and $\hat{\mathbf q}+\mathbf p$ has the same norm as $\mathbf p$, and $\mathbf q$ is of type $I_{-}$ if $\hat{\mathbf q}$ is inside the open disk of radius $\mathbf p$ and $\hat{\mathbf q}-\mathbf p$ has the same norm as $\mathbf p$. For example, see Figure \ref{Classes of points}, let $\bp=(3,1)$. Then $\hat{\bq}=(-2,3)$ is of type $0$, 
$\hat{\bq}=(-1,2)$ is of type $I_{0}$, $\hat{\bq}=(0,-2)$ is of type $I_{+}$, $\hat{\bq}=(2,-2)$ is of type $I_{-}$ and $\hat{\bq}=(-1,1)$ is of type $II$. Notice that if $\mathbf q$ is of type $I_{0}$, then $- \mathbf q$ is also necessarily of type $I_{0}$. Also, if $\mathbf q$ is of type $I_{+}$, then $-\mathbf q$ is of type $I_{-}$ and vice versa, if $\mathbf q$ is of type $I_{-}$, then $-\mathbf q$ is of type $I_{+}$.

\end{remark}
\begin{remark}\label{rho1}
If $\bq$ is of type $I_{0}$, then $\|\mathbf q\| < \|\mathbf p\|$ and $\|\mathbf q+n\mathbf p\| > \|\mathbf p\|$ for all $n \neq 0$. Consequently, see \eqref{rhon},   $\rho_{0}<0$ and $\rho_{n}>0$ for all $n \neq 0$. If $\bq$ is of type $I_{+}$, then $\|\mathbf q\| < \|\mathbf p\|$, $\|\mathbf q+\mathbf p\| = \|\mathbf p\|$ and $\|\mathbf q+n \mathbf p\| > \|\mathbf p\|$ for all $n \neq 0,1$. Therefore $\rho_{0}<0$ and $\rho_{1}=0$ and $\rho_{n}>0$ for all $n \neq 0,1$. Similarly, if $\bq$ is of type $I_{-}$, then $\rho_{0}<0$ and $\rho_{-1}=0$ and $\rho_{n}>0$ for all $n \neq 0,-1$.
\end{remark}

In what follows, %
for a point $\bq$ of type $I$, 
we will drop hat in the notation $\hat{\bq}$ and assume that $\bq \in \bbZ^{2}$ satisfies $\|\bq\| < \|\bp\|$. Our main goal is to prove instability of type $I$ flows and in what follows we shall assume that $\mathbf q$ is of type $I$. In Section 2 below, we first study the case when $\mathbf q$ is of type $I_{0}$ and then in Subsection 2.1 we outline changes to the main results when $\mathbf q$ is of type $I_{+}$ and $I_{-}$.

\begin{notation}\label{cf-notation}
In what follows we shall sometimes abbreviate a continued fraction of the form 
\begin{equation*}
f=\cfrac{1}{
a_{1}+\cfrac{1}{
a_{2}+\dots}}
\end{equation*} 
by the expression $[a_{1};a_{2};\ldots]$ and use $f^{k}$ to denote the k-th truncation of the continued fraction above, i.e., 
\begin{equation*}
f^{k}=\cfrac{1}{ a_{1}+\cfrac{1}{a_{2}+\cfrac{1}{\ddots+\cfrac{1}{ a_{k}}}}},
\end{equation*}
which we shall abbreviate by $[a_{1};a_{2};\ldots;a_{k}]$.
\end{notation}
\section{The main instability theorem}
Consider the eigenvalue problem 
\begin{equation}
L_{B,\mathbf q}(w_{n})=\lambda (w_{n}).
\end{equation} 
This can be rewritten, see \eqref{Lbqnew}, as
\begin{equation}\label{evns}
\rho_{n-1}w_{n-1}-\rho_{n+1}w_{n+1}=\lambda w_{n}+ \nu \|\mathbf q + n \mathbf p\|^{2}w_{n}.
\end{equation}
Assuming first that our steady state is of type $I_{0}$ (recall Remark \ref{rho1} above, $\rho_{0}<0$ and $\rho_{n}>0$ for all $n \neq 0$), our main goal is to construct solutions to \eqref{evns} where the eigenvector sequence $(w_{n})$ is such that $n^{2}(w_{n})$ belongs to the space $\ell_{2}(\mathbb Z)$ and the eigenvalue $\lambda$ has positive real part.

Let us denote by $c_{n}=  \|\mathbf q + n \mathbf p\|^{2}$ and, also, let us denote by 
\begin{equation}\label{an}
a_{n}:=a_{n}(\lambda, \nu)= \frac{\lambda + \nu c_{n}}{ \rho_{n}} =\frac{\lambda + \nu \|\mathbf q + n \mathbf p\|^{2}}{\rho_{n}},
\end{equation}
and note that $a_{n}(\lambda,\nu) \to \infty$ as $n \to \pm \infty$ and $a_{0}(\lambda,\nu)= \frac{\lambda+\nu \|\mathbf q\|^{2}}{\rho_{0}}$.
Assuming the eigenvalue equation \eqref{evns} has a positive eigenvalue $\lambda>0$, if we let $z_{n}=\rho_{n}w_{n}$, then, we can rewrite \eqref{evns} as 
\begin{equation}\label{znformula}
z_{n-1}-z_{n+1}= a_{n}z_{n}.
\end{equation}
Assuming $z_{n} \neq 0$ for every $n$, if we denote $u_{n}=z_{n-1}/z_{n}$, the above equation reduces to $u_{n}-\frac{1}{u_{n+1}}=a_{n},  \quad n \in \mathbb Z$, 
rewriting which we obtain
\begin{equation}\label{unformula1}
u_{n}=a_{n}+\frac{1}{u_{n+1}}, \quad n \in \mathbb Z.
\end{equation}
Iterating this equation above forwards for $n \geq 0$, consider the following continued fraction defined for $n \geq 0$
\begin{equation}\label{un1}
u_{n}^{(1)}(\lambda,\nu):= a_{n}+\cfrac{1}{
a_{n+1}+\cfrac{1}{
a_{n+2}+\dots}}, \, \quad n=0,1,2,\ldots.
\end{equation}
Since $\sum_{n} a_{n}$ diverges, for each $n \geq 0$ the continued fraction above converges, by Van Vleck Theorem, see \cite[Theorem 4.29]{JT}.  Moreover, using the fact that $a_{n}(\lambda,\nu) \to \infty$ as $n \to \infty$, one also obtains, see Lemma \ref{ulemma} Item (2) below, that 
\begin{equation}\label{un1inf}
\lim_{n \to \infty} u_{n}^{(1)}(\lambda,\nu) = \infty.
\end{equation}
One can prove, see \cite[pp. 2063-2064]{DLMVW20}, that the continued fraction expression $u_{n}^{(1)}(\lambda,\nu)$ in \eqref{un1} obtained by iterating \eqref{unformula1} above is equal to $u_{n}$ given by \eqref{unformula1} for every $n$. 
That is, we have (recall notation defined in Notation \ref{cf-notation}) 
\begin{equation}
u_{n}=u_{n}^{(1)}(\lambda,\nu) =a_{n}+ [a_{n+1};a_{n+2};\ldots], \quad n \geq 0.
\end{equation}
Similarly, \eqref{unformula1} can be rewritten as 
\begin{equation}\label{unformula2}
u_{n+1}=\frac{-1}{a_{n}-u_{n}}, \quad n \in \mathbb Z. 
\end{equation}
Iterating this for $n \leq 0$, consider the continued fractions 
\begin{equation}\label{un2}
u_{n}^{(2)}(\lambda,\nu):= \cfrac{-1}{
a_{n-1}+\cfrac{1}{
a_{n-2}+\dots}}, \quad n=0,-1,-2,\ldots.
\end{equation}
Again, by Van Vleck theorem, the continued fraction expressions converge for all $n \leq -1$. Moreover, using the fact that $a_{n}(\lambda,\nu) \to \infty$ as $n \to -\infty$, we obtain that 
\begin{equation}\label{un2inf}
\lim_{n \to -\infty} u_{n}^{(2)}(\lambda,\nu) = 0.
\end{equation}
As before, the proof that the continued fraction expression $u_{n}^{(2)}(\lambda,\nu)$ defined by the second equality above is equal to $u_{n}$ is given in \cite[pp. 2063-2064]{DLMVW20}. That is, 
\begin{equation}
u_{n}=u_{n}^{(2)}(\lambda,\nu)= - [a_{n-1};a_{n-2};\ldots], \quad n \leq 0.
\end{equation}
Since the expressions for $u_{0}$, given, respectively, by equations \eqref{un1} and \eqref{un2} must match, we must have $u_{0}^{(1)}(\lambda,\nu)=u_{0}^{(2)}(\lambda,\nu)$, i.e.,
\begin{equation}\label{un12}
u_{0}^{(1)}(\lambda,\nu)= a_{0}+[a_{1};a_{2};\ldots]=u_{0}^{(2)}(\lambda,\nu)=-[a_{-1};a_{-2};\ldots].
\end{equation}
Denote 
\begin{equation}\label{f}
f(\lambda, \nu)= [a_{1}(\lambda, \nu);a_{2}(\lambda, \nu);\ldots]
\end{equation}
and
\begin{equation}\label{g}
g(\lambda, \nu)=[a_{-1}(\lambda, \nu);a_{-2}(\lambda, \nu);\ldots].
\end{equation}
Notice that
\begin{equation}\label{fgu012}
f(\lambda,\nu)=u_{0}^{(1)}(\lambda,\nu)-a_{0}(\lambda,\nu), \quad g(\lambda,\nu)=-u_{0}^{(2)}(\lambda,\nu).
\end{equation}
Equation \eqref{un12} is equivalent to 
\begin{equation}\label{fg}
a_{0}(\lambda, \nu)+f(\lambda, \nu)+g(\lambda, \nu)=0.
\end{equation}
Thus the existence of a positive eigenvalue to equation \eqref{evns} implies the existence of a root $\lambda>0$ to equation \eqref{fg}. One can also run the construction backwards, see Theorem \ref{mainthm} below. Beginning with a positive root $\lambda >0$ to equation \eqref{fg}, one can define (using the expressions for $a_{n}$ in \eqref{an}) continued fractions $u_{n}^{(1)}(\lambda,\nu)$ and $u_{n}^{(2)}(\lambda,\nu)$ and use these to construct eigenvectors $(w_{n})$ so that the eigenvalue equation \eqref{evns} has an eigenvalue $\lambda > 0$. Thus, the eigenvalue problem \eqref{evns} has an unstable eigenvalue $\lambda > 0$ if and only if equation \eqref{fg} has a root $\lambda > 0$.
\begin{remark}\label{root}
Our strategy in proving that \eqref{fg} has a positive root $\lambda >0$ is as follows. Considered as a function of $\lambda \geq 0$, $-a_{0}(\lambda,\nu)= -(\lambda+\nu \|\mathbf q\|^{2})/\rho_{0}$ is a straight line of positive slope $-1/\rho_{0}$ (recall $\rho_{0}<0$) and positive $y$-intercept $-\nu \|\mathbf q\|^{2}/\rho_{0}$. Also, see Lemma \ref{ulemma} and Lemma \ref{root} below, for $\lambda > 0$ $f(\lambda, \nu)$ and $g(\lambda, \nu)$ are non negative, continuous functions in $\lambda$ and $f(\lambda, \nu) $ and $g(\lambda, \nu)$ go to $0$ as $\lambda \to \infty$. Moreover, $f(\lambda,\nu)$ and $g(\lambda,\nu)$ converge, respectively, to the finite, positive numbers $f(0,\nu)$ and $g(0,\nu)$ as $\lambda \to 0^{+}$, see Lemma \ref{ff0} below. This then means that in the $\lambda$-plane, the continuous curves $f(\lambda,\nu)+g(\lambda,\nu)$ and $-a_{0}(\lambda,\nu)$ will intersect at a point $\lambda > 0$ provided that  $f(0,\nu)+g(0,\nu)$ is greater than the $y$-intercept of $a_{0}(\lambda,\nu)$, $-\nu \|\mathbf q\|^{2}/\rho_{0}$. This inequality holds for small enough $\nu$, see Lemma \ref{nu0} below. Our strategy in proving that $f(0,\nu)+g(0,\nu)> -\nu \|\mathbf q\|^{2}/\rho_{0}$ for $\nu$ small enough is as follows. Considered as functions of the variable $\nu$, $f(0,\nu)+g(0,\nu)$ are continuous in $\nu$ and are bounded below by the even truncations $f(0,\nu)^{2k}+g(0,\nu)^{2k}$ for every fixed $k$. The straight line $-\nu \|\mathbf q\|^{2}/\rho_{0}$ passes through the origin and has positive slope $-\|\mathbf q\|^{2}/\rho_{0}$ and we show in Lemma \ref{even-trunc-properties} below that the even truncations $f(0,\nu)^{2k}+g(0,\nu)^{2k}$ are non-negative on $[0,+\infty)$ and satisfy the limits $\lim_{\nu \to 0} f(0,\nu)^{2k} = \lim_{\nu \to 0} g(0,\nu)^{2k}=0$ and $\lim_{\nu \to +\infty} f(0,\nu)^{2k} = \lim_{\nu \to 0} g(0,\nu)^{2k}=0$. Moreover, the slopes at $0$ (i.e., derivative with respect to $\nu$ at $0$) $f'(0,\nu)^{2k}|_{\nu=0}$ and $g'(0,\nu)^{2k}|_{\nu=0}$ are greater than the slope of the line $-\nu \|\mathbf q\|^{2}/\rho_{0}$ for large enough $k$. This then means, see Lemma \ref{nu0} below, that there will exist a $\nu_{0}>0$ such that for all $\nu \in (0,\nu_{0})$, there holds  $f(0,\nu)+g(0,\nu)>-\nu \|\mathbf q\|^{2}/\rho_{0}$. See Figure \ref{ENS} for a numerical illustration described in this Remark.
\end{remark}

We summarize the properties of the continued fractions $u_{n}^{(1)}(\lambda,\nu)$ and $u_{n}^{(2)}(\lambda,\nu)$ in the following Lemma.
\begin{lemma}\label{ulemma}
Fix $\nu>0$, and consider the continued fractions $u_{n}^{(1)}(\lambda,\nu)$ and $u_{n}^{(2)}(\lambda,\nu)$ defined by equations \eqref{un1} and \eqref{un2} for all $\lambda >0$. $u_{n}^{(1)}(\lambda,\nu)$ and $u_{n}^{(2)}(\lambda,\nu)$ satisfy the following properties:
\begin{enumerate}
\item[(1)] $u_{n}^{(1)}(\lambda,\nu)$ and $u_{n}^{(2)}(\lambda,\nu)$ are convergent continued fractions and the functions $u_{n}^{(1)}(\cdot,\nu)$ and $u_{n}^{(2)}(\cdot,\nu)$ are continuous in $\lambda$. 
\item[(2)] There exist limits
\begin{equation}\label{limitsuninfty}
\lim_{n \to \infty} u_{n}^{(1)}(\lambda,\nu)=+\infty, \quad \lim_{n \to -\infty} u_{n}^{(2)}(\lambda,\nu)=0, \quad \lambda >0.
\end{equation}
\item[(3)] For some $0<q < 1$ and $C>0$, the following hold 
\begin{equation}\label{u1ineq}
(|u_{1}^{(1)}(\lambda,\nu)u_{2}^{(1)}(\lambda,\nu)\ldots u_{n}^{(1)}(\lambda,\nu)|)^{-1} \leq Cq^{n},  \mbox{ for all }n \geq 0,
\end{equation}
\begin{equation}\label{u2ineq}
(|u_{n}^{(2)}(\lambda,\nu)\ldots u_{-2}^{(2)}(\lambda,\nu)u_{-1}^{(2)}(\lambda,\nu)u_{0}^{(2)}(\lambda,\nu)|) \leq Cq^{-n}, \mbox{ for all } n \leq -1.
\end{equation}
\end{enumerate}
\end{lemma}
\begin{proof}
(1) This follows from the Van Vleck and the Stieltjes-Vitali theorems, see \cite[Theorem 4.29 and Theorem 4.30]{JT} and the facts that $\sum_{n \geq 0} a_{n}$ and $\sum_{n \geq 0} a_{-n}$ diverge. These two theorems guarantee, in particular, that the continued fractions $u_{n}^{(1)}(\lambda,\nu)$ and $u_{n}^{(2)}(\lambda,\nu)$ converge for every $\lambda >0$ and moreover, the functions $u_{n}^{(1)}(\cdot,\nu)$ and $u_{n}^{(2)}(\cdot,\nu)$ are continuous in $\lambda$.

(2) Recall formula \eqref{an}, $a_{n} =\frac{\lambda + \nu \|\mathbf q + n \mathbf p\|^{2}}{\rho_{n}}$ and note that $a_{n}(\lambda,\nu) \to \infty$ as $n \to \pm \infty$. Fix $n> 0$. $u_{n}^{(1)}(\lambda,\nu)$ is positive and bounded below by $a_{n}(\lambda,\nu)$ which is a positive sequence for $n > 0$. Moreover $a_{n}(\lambda,\nu) \to \infty$ as $n \to +\infty$ implying that $u_{n}^{(1)}(\lambda,\nu) \to +\infty$ as $n \to +\infty$. 

Notice that $u_{n}^{(2)}(\lambda,\nu)$ satisfies the inequalities
\begin{equation}
\frac{-1}{a_{n-1}(\lambda,\nu)} \leq u_{n}^{(2)}(\lambda,\nu) \leq \frac{1}{a_{n-1}(\lambda,\nu)}.
\end{equation} 
Using the fact that $a_{n}(\lambda,\nu) \to +\infty$ as $n \to -\infty$ implies that $u_{n}^{(2)}(\lambda,\nu) \to 0$ as $n \to - \infty$.

(3) The estimates \eqref{u1ineq} and \eqref{u2ineq} are a consequence of the limits \eqref{limitsuninfty} in item (2). Indeed, since $u_{n}^{(1)}(\lambda,\nu) \to +\infty$ as $n \to \infty$, given a $K>1$, there exists a constant $N=N(K)$ such that if $n > N$, $u_{n}^{(1)}(\lambda,\nu) > K$. One thus has
\begin{align*}
&u_{1}^{(1)}(\lambda,\nu)u_{2}^{(1)}(\lambda,\nu)\ldots u_{n}^{(1)}(\lambda,\nu)\\&= u_{1}^{(1)}(\lambda,\nu)u_{2}^{(1)}(\lambda,\nu)\ldots u_{N}^{(1)}(\lambda,\nu)u_{N+1}^{(1)}(\lambda,\nu)\ldots u_{n}^{(1)}(\lambda,\nu)\\&
 \geq u_{1}^{(1)}(\lambda,\nu)u_{2}^{(1)}(\lambda,\nu)\ldots u_{N}^{(1)}(\lambda,\nu) K^{n-N-1} =C_{1} K^{n}, \numberthis \label{ineqproof}
\end{align*}
where the constant $C_{1}= u_{1}^{(1)}(\lambda,\nu)u_{2}^{(1)}(\lambda,\nu)\ldots u_{N}^{(1)}(\lambda,\nu) K^{-N-1}$. Put $q=1/K$ and taking the reciprocal of \eqref{ineqproof}, we obtain \eqref{u1ineq}.

Similarly, since $u_{n}^{(2)}(\lambda,\nu) \to 0$ as $n \to -\infty$, given $q \in (0,1)$, there exists an $N=N(q)$ such that if $n < N$, $|u_{n}^{(2)}(\lambda,\nu)| < q$. As before, we estimate, for $n<0$
\begin{align*}
&|u_{0}^{(2)}(\lambda,\nu)u_{-1}^{(2)}(\lambda,\nu)\ldots u_{n}^{(2)}(\lambda,\nu)|\\&= |u_{0}^{(2)}(\lambda,\nu)u_{-1}^{(2)}(\lambda,\nu)\ldots u_{N}^{(2)}(\lambda,\nu)u_{N-1}^{(2)}(\lambda,\nu)\ldots u_{n}^{(2)}(\lambda,\nu)|\\&
 \leq |u_{0}^{(2)}(\lambda,\nu)u_{-1}^{(2)}(\lambda,\nu)\ldots u_{N}^{(2)}(\lambda,\nu)| q^{n-N-1} =C q^{n}, \numberthis \label{ineqproof2}
\end{align*}
where $C = |u_{1}^{(2)}(\lambda,\nu)u_{2}^{(2)}(\lambda,\nu)\ldots u_{N}^{(2)}(\lambda,\nu)| q^{-N-1}$ which proves \eqref{u2ineq}.
\end{proof}

Notice that $f(0,\nu)$ and $g(0,\nu)$ are given by the expressions
\begin{equation}\label{f0}
f(0, \nu)= [a_{1}(0, \nu);a_{2}(0, \nu);\ldots]
\end{equation}
and
\begin{equation}\label{g0}
g(0, \nu)=[a_{-1}(0, \nu);a_{-2}(0, \nu);\ldots].
\end{equation}
Moreover, these continued fractions converge by Van Vleck theorem since %
the series $\sum_{n}a_{n}(0,\nu)$ and $\sum_{n}a_{-n}(0,\nu)$ diverge. Also, the following is true.
\begin{lemma}\label{ff0}
The continued fractions $f(\lambda,\nu)$ and $g(\lambda,\nu)$ converge, respectively, to the continued fractions $f(0,\nu)$ and $g(0,\nu)$ as $\lambda$ goes to $0^{+}$. That is,
\begin{equation}
\lim_{\lambda \to 0^{+}} f(\lambda,\nu) = f(0,\nu), \quad \lim_{\lambda \to 0^{+}} g(\lambda,\nu) = g(0,\nu).
\end{equation}
\end{lemma}
\begin{proof}
Let us prove that $\lim_{\lambda \to 0^{+}} f(\lambda,\nu) = f(0,\nu)$, the proof that\\ $\lim_{\lambda \to 0^{+}} g(\lambda,\nu) = g(0,\nu)$ is similar. It follows, from standard facts of continued fractions, see for example \cite[Chapter 2]{JT}, that the odd $k^\text{th}$ truncations $f(\lambda,\nu)^{2k+1}$ form a monotonically decreasing sequence and the even $k^\text{th}$ truncations $f(\lambda,\nu)^{2k}$ form a monotonically increasing sequence and $f(\lambda,\nu)$ is sandwiched in between these. That is, we have, for every $k \geq 1$,
\begin{equation}\label{fconv}
f(\lambda,\nu)^{2k-2} \leq f(\lambda,\nu)^{2k} \leq f(\lambda,\nu) \leq f(\lambda,\nu)^{2k+1} \leq f(\lambda,\nu)^{2k-1}.
\end{equation}
Similar facts hold for the continued fraction $f(0,\nu)$ and its $k^{\text{th}}$ truncations $f(0,\nu)^{k}$.

Given $\epsilon>0$, we would like to find a $\delta=\delta(\epsilon)$ such that, for all $0< \lambda < \delta$, there holds
\begin{equation}
|f(\lambda,\nu)-f(0,\nu)| < \epsilon.
\end{equation}
Since $f(0,\nu)^{k}$ converges to $f(0,\nu)$ as $k \to + \infty$, fixing $\epsilon>0$, there exists a $K=K(\epsilon)$ large enough such that the following estimate holds
\begin{equation}\label{ff02}
II:=|f(0,\nu)^{2K+1}-f(0,\nu)^{2K}| < \frac{\epsilon}{2}.
\end{equation}
Since $f(\lambda,\nu)^{2K+1}$ and $f(0,\nu)^{2K+1}$ are finite fractions, where $a_{n}(\lambda,\nu) \to a_{n}(0,\nu)$ for every $n$, they are continuous in $\lambda$. Moreover, the finite fractions $f(\lambda,\nu)^{2K+1}$ and $f(0,\nu)^{2K+1}$ are ratios of polynomials in $\lambda$ that are positive for every non-negative $\lambda \geq 0$ and thus have no roots in the interval $[0,A]$ for any $A >0$. We thus have, for a fixed $\epsilon>0$ and fixed $K=K(\epsilon)$ chosen above so that estimate \eqref{ff02} holds, there exists a $\delta_{1}=\delta_{1}(K,\epsilon)$ such that if $0<\lambda < \delta_{1}$ then
\begin{equation}\label{ff01}
I:=|f(\lambda,\nu)^{2K+1}-f(0,\nu)^{2K+1}| < \frac{\epsilon}{2},
\end{equation}
and similarly there exists a $\delta_{2}=\delta_{2}(K,\epsilon)$ such that if  $0<\lambda < \delta_{2}$ then
\begin{equation}\label{ff04}
IV:=|f(\lambda,\nu)^{2K}-f(0,\nu)^{2K}| < \frac{\epsilon}{2}.
\end{equation}
Now choose  $\delta = \delta (\epsilon,K)=\mbox{ min }\{\delta_{1},\delta_{2}\}$ so that both \eqref{ff01} and \eqref{ff04} hold. We shall now estimate $|f(\lambda,\nu)-f(0,\nu)|$. There are two possibilities. Either $f(\lambda,\nu)> f(0,\nu)$ or $f(\lambda,\nu)< f(0,\nu)$ (we need not consider the case when $f(\lambda,\nu)=f(0,\nu)$ since $|f(\lambda,\nu)-f(0,\nu)|=0$).

Assuming that $f(\lambda,\nu)>f(0,\nu)$, using \eqref{fconv}, we have, 
\begin{equation}\label{ff01main}
|f(\lambda,\nu)-f(0,\nu)| \leq |f(\lambda,\nu)^{2K+1}-f(0,\nu)^{2K}|,
\end{equation}
where we used the facts, see \eqref{fconv}, that $f(\lambda,\nu)^{2K+1} \geq f(\lambda,\nu)$ and $f(0,\nu)^{2K} \leq f(0,\nu)$. Assuming that  $0<\lambda < \delta$, we estimate \eqref{ff01main} by
\begin{align}
|f(\lambda,\nu)^{2K+1}-f(0,\nu)^{2K}| &\leq |f(\lambda,\nu)^{2K+1}-f(0,\nu)^{2K+1}|\nonumber \\&+|f(0,\nu)^{2K+1}-f(0,\nu)^{2K}|  =I+II< \frac{\epsilon}{2}+\frac{\epsilon}{2}=\epsilon, 
\end{align}
using equations \eqref{ff02}and \eqref{ff01}, where $0<\lambda < \delta$.

Now, suppose $f(\lambda,\nu)<f(0,\nu)$, using \eqref{fconv}, we have, 
\begin{equation}\label{ff01main2}
|f(\lambda,\nu)-f(0,\nu)| \leq |f(0,\nu)^{2K+1}-f(\lambda,\nu)^{2K}|,
\end{equation}
where we used the facts, see \eqref{fconv}, that $f(0,\nu)^{2K+1} \geq f(0,\nu)$ and $f(\lambda,\nu)^{2K} \leq f(\lambda,\nu)$. Assuming that  $0<\lambda < \delta$, we estimate \eqref{ff01main2} by
\begin{align}
|f(0,\nu)^{2K+1}-f(\lambda,\nu)^{2K}| \leq& |f(0,\nu)^{2K+1}-f(0,\nu)^{2K}|+|f(0,\nu)^{2K}-f(\lambda,\nu)^{2K}| \nonumber \\& =II+IV<\frac{\epsilon}{2}+\frac{\epsilon}{2}=\epsilon,
\end{align}
using equations \eqref{ff02}and \eqref{ff04}, where  $0<\lambda < \delta$.
\end{proof}
Recall, from \eqref{an}, the formulas $a_{n}(0,\nu)= \frac{\nu\|\mathbf q+n\mathbf p\|^{2}}{\rho_{n}}$, where $\rho_{n}=1-\frac{\|\mathbf p\|^{2}}{\|\mathbf q+n\mathbf p\|^{2}}=\frac{\|\mathbf q+n\mathbf p\|^{2}-\|\mathbf p\|^{2}}{\|\mathbf q+n\mathbf p\|^{2}}$. Thus,
\begin{equation}
a_{n}(0,\nu)=\frac{\nu \|\mathbf q+n\mathbf p\|^{2}}{\rho_{n}}=\frac{\nu \|\mathbf q+n\mathbf p\|^{4}}{\|\mathbf q+n\mathbf p\|^{2}-\|\mathbf p\|^{2}}.
\end{equation}
Let us introduce the following notation. Let 
\begin{equation}\label{bn}
b_{n}=\frac{a_{n}(0,\nu)}{\nu} = \frac{ \|\mathbf q+n\mathbf p\|^{4}}{\|\mathbf q+n\mathbf p\|^{2}-\|\mathbf p\|^{2}}.
\end{equation}
The continued fractions $f(0,\nu)$ and $g(0,\nu)$ (recall \eqref{f0} and \eqref{g0}) can be rewritten as
\begin{equation}\label{f0b}
f(0, \nu)= [\nu b_{1};\nu b_{2};\ldots]
\end{equation}
and
\begin{equation}\label{g0b}
g(0, \nu)=[\nu b_{-1};\nu b_{-2};\ldots].
\end{equation}
Let $f(0,\nu)^{k}$ denote the k-th truncation of the continued fraction $f(0,\nu)$. That is,
\begin{equation}
f(0,\nu)^{k}=[\nu b_{1};\nu b_{2};\ldots;\nu b_{k}].
\end{equation}
And similarly, let $g(0,\nu)^{k}$ denote the k-th truncation of the continued fraction $g(0,\nu)$. In what follows we shall need certain properties of the even truncations $f(0,\nu)^{2k}$ and $g(0,\nu)^{2k}$ of the continued fractions $f(0,\nu)$ and $g(0,\nu)$. One can see, by direct calculation, that
\begin{equation}\label{f2}
f(0,\nu)^{2}=\cfrac{1}{\nu b_{1}+\cfrac{1}{\nu b_{2}}}=\frac{ b_{2}\nu}{1+\nu^{2}b_{1}b_{2}}
\end{equation}
and
\begin{equation}\label{f4}
f(0,\nu)^{4}=[\nu b_{1};\nu b_{2};\nu b_{3};\nu b_{4}]=\frac{(b_{2}+b_{4})\nu+b_{2}b_{3}b_{4}\nu^{3}}{1+(b_{1}b_{2}+b_{1}b_{4}+b_{3}b_{4})\nu^{2}+b_{1}b_{2}b_{3}b_{4}\nu^{4}}.
\end{equation}
We prove in Lemma \ref{even-trunc-properties} below that the even truncations $f(0,\nu)^{2k}$ are of the form
\begin{equation}\label{f2k}
f(0,\nu)^{2k}=\frac{(b_{2}+b_{4}+\ldots +b_{2k})\nu+p_{3}\nu^{3}+p_{5}\nu^{5}+\ldots+p_{2k-1}\nu^{2k-1}}{1+q_{2}\nu^{2}+q_{4}\nu^{4}+\ldots+q_{2k}\nu^{2k}},
\end{equation}
where the coefficients $p_{3},p_{5},\ldots, p_{2k-1}, q_{2},q_{4},\ldots q_{2k}$ are all positive. Similarly, $g(0,\nu)^{2k}$ is given by the formula
\begin{equation}\label{g2k}
g(0,\nu)^{2k}=\frac{(b_{-2}+b_{-4}+\ldots +b_{-2k})\nu+r_{3}\nu^{3}+r_{5}\nu^{5}+\ldots+r_{2k-1}\nu^{2k-1}}{1+s_{2}\nu^{2}+s_{4}\nu^{4}+\ldots+s_{2k}\nu^{2k}},
\end{equation}
where the coefficients $r_{3},r_{5},\ldots, r_{2k-1}, s_{2},s_{4},\ldots s_{2k}$ are all positive. 
\begin{lemma}\label{even-trunc-properties}
Fix $k>0$, and consider the truncated continued fractions $f(0,\nu)^{2k}$ and $g(0,\nu)^{2k}$. They satisfy the following properties.

(P) $f(0,\nu)^{2k}$ and $g(0,\nu)^{2k}$ are ratios of polynomials in the variable $\nu$ given by formulas \eqref{f2k} and \eqref{g2k}, where the coefficients are all positive. Moreover, the following facts hold.

(1) $f(0,\nu)^{2k}$ and $g(0,\nu)^{2k}$ are non-negative, continuous and differentiable functions in the interval $[0,+\infty)$. 

(2) There exist limits \[\lim_{\nu \to 0^{+}}f(0,\nu)^{2k} = \lim_{\nu \to 0^{+}}g(0,\nu)^{2k}=0\] and \[\lim_{\nu \to +\infty}f(0,\nu)^{2k}=\lim_{\nu \to +\infty}g(0,\nu)^{2k}=0.\]

(3) The following formulae hold: \[f'(0,\nu)^{2k}|_{\nu=0}= b_{2}+b_{4}+\ldots b_{2k}\] and \[g'(0,\nu)^{2k}|_{\nu=0}= b_{-2}+b_{-4}+\ldots b_{-2k}.\]
\end{lemma}
\begin{proof}
We prove the facts for $f(0,\nu)^{2k}$, those for $g(0,\nu)^{2k}$ are similar. The proof of Property (P) is by induction on $k$. When $k=1$, the result holds for $f(0,\nu)^{2}$ by inspecting formula \eqref{f2}. Suppose the result holds for $k> 1$. We need to show that the result holds for $k+1$. Consider the truncation $f(0,\nu)^{2k+2}$. This is given by the formula
\begin{equation}\label{f2k+2}
f(0,\nu)^{2k+2}=\cfrac{1}{\nu b_{1}+\cfrac{1}{\nu b_{2}+h_{2k}(0,\nu)}},
\end{equation} 
where $h_{2k}(0,\nu)$ is given by the following formula (recall Notation \ref{cf-notation})
\begin{equation}
h_{2k}(0,\nu)= [\nu b_{3};\nu b_{4};\nu b_{5};\ldots;\nu b_{2k+2}].
\end{equation}
By the induction hypothesis, $h_{2k}(0,\nu)$ is a polynomial of order $2k$ and is given by
\begin{equation}\label{h2k}
h_{2k}(0,\nu)= \frac{(b_{4}+b_{6}+\ldots +b_{2k+2})\nu+p_{3}\nu^{3}+p_{5}\nu^{5}+\ldots+p_{2k-1}\nu^{2k-1}}{1+q_{2}\nu^{2}+q_{4}\nu^{4}+\ldots+q_{2k}\nu^{2k}},
\end{equation}
where the coefficients are all positive. Denote the numerator and denominator in \eqref{h2k} by $n_{2k}(\nu)$ and $d_{2k}(\nu)$ respectively, and plugging \eqref{h2k} into \eqref{f2k+2} we obtain
\begin{equation}\label{f2k+2new}
f(0,\nu)^{2k+2}=\frac{b_{2}\nu d_{2k}(\nu)+n_{2k}(\nu)}{b_{1}b_{2}\nu^{2}d_{2k}(\nu)+b_{1}\nu n_{2k}(\nu)+d_{2k}(\nu)}.
\end{equation}
The numerator and denominator of \eqref{f2k+2new} are given by the following formulas
\begin{align*}
&b_{2}\nu(1+q_{2}\nu^{2}+q_{4}\nu^{4}+\ldots+q_{2k}\nu^{2k})+(b_{4}+b_{6}+\ldots +b_{2k+2})\nu\\&+p_{3}\nu^{3}+p_{5}\nu^{5}+\ldots+p_{2k-1}\nu^{2k-1} \numberthis \label{nr}
\end{align*} 
and
\begin{align*}
&(1+b_{1}b_{2}\nu^{2})(1+q_{2}\nu^{2}+q_{4}\nu^{4}+\ldots+q_{2k}\nu^{2k})+\\&b_{1}\nu ((b_{4}+b_{6}+\ldots +b_{2k+2})\nu+p_{3}\nu^{3}+p_{5}\nu^{5}+\ldots+p_{2k-1}\nu^{2k-1}) \numberthis \label{dr}.
\end{align*}
Expanding the numerator in \eqref{nr}, we observe that the coefficient of $\nu$ is given by $b_{2}+b_{4}+\ldots b_{2k+2}$ and thus the formula for \eqref{f2k+2new} has the same form as \eqref{f2k}. Property (P) thus holds for $k+1$ and is thus proven. Properties (1), (2) directly follow from the polynomial formula in equation \eqref{f2k}. Using the formulae \eqref{f2k} and \eqref{g2k} and the fact that $f(0,0)^{2k}=g(0,0)^{2k=0}$, the derivative formula $f'(0,\nu)^{2k}=\lim_{\nu \to 0^{+}}\frac{f(0,\nu)-f(0,0)}{\nu}$ yields the formulae in (3).
\end{proof}
By Remark \ref{root}, equation \eqref{fg}, 
\begin{equation}\label{fg-2}
a_{0}(\lambda, \nu)+f(\lambda, \nu)+g(\lambda, \nu)=0,
\end{equation}
has a positive root $\lambda > 0$ for those values of $\nu$ for which the following inequality holds
\begin{equation}\label{ineqf0g0}
f(0,\nu)+g(0,\nu) > -a_{0}(0,\nu)= -\frac{\nu\|\mathbf q\|^{2}}{\rho_{0}}.
\end{equation}
In the Lemma below, we use the following fact from one variable real analysis that was used in \cite{FSV97}, see page 208 therein. Let $f_{1}:[0,+\infty)$ be at least continuously differentiable, positive on $(0,+\infty)$ with $f_{1}(0)=0$ and let $f_{2}(x)=ax$ be the straight line with positive slope $a$. If $f_{1}'(0) > a$ (the derivative taken from the right), then there exists a $\delta > 0$ such that $f_{1}(x)> f_{2}(x)=ax$ for all $x$ in the interval $(0,\delta)$. %
\begin{lemma}\label{nu0}
There exists a $\nu_{0}>0$ such that for all $\nu$ in the interval $(0,\nu_{0})$ the inequality in equation \eqref{ineqf0g0} is satisfied.
\end{lemma}
\begin{proof}
It follows by standard facts from the theory of continued fractions that $f(0,\nu)$ and $g(0,\nu)$ are bounded below by the even truncations $f(0,\nu)^{2k}$ and $g(0,\nu)^{2k}$ for every $k>0$. It thus suffices to check that there exists a (finite) $k>0$ and a $\nu_{0}=\nu_{0}(k)$ such that for all $\nu$ in the interval $(0,\nu_{0})$ the inequality 
\begin{equation}\label{f0g0a0}
f(0,\nu)^{2k}+g(0,\nu)^{2k} > -a_{0}(0,\nu)
\end{equation}
holds. As a function of $\nu$, $-a_{0}(0,\nu)$ is a straight line with positive slope $-\frac{\|\mathbf q\|^{2}}{\rho_{0}}$. In particular $-a_{0}(0,\nu) \to +\infty$ as $\nu \to +\infty$. By Lemma \ref{even-trunc-properties} we know that for each fixed $k$, $f(0,\nu)^{2k}$ and $g(0,\nu)^{2k}$ are continuous, positive on $(0,+\infty)$ with limits $\lim_{\nu \to 0^{+}} f(0,\nu)^{2k} = \lim_{\nu \to 0^{+}} g(0,\nu)^{2k} =0$ and $\lim_{\nu \to +\infty} f(0,\nu)^{2k}=\lim_{\nu \to +\infty} g(0,\nu)^{2k}=0$. Moreover, the slopes at $0$ satisfy the formula
\begin{equation}
f'(0,\nu)^{2k}|_{\nu=0}= b_{2}+b_{4}+\ldots b_{2k}, \quad g'(0,\nu)^{2k}|_{\nu=0}= b_{-2}+b_{-4}+\ldots b_{-2k}.
\end{equation} 
If we can show that the slope of $f'(0,\nu)^{2k}|_{\nu=0}+g'(0,\nu)^{2k}|_{\nu=0}$ is greater than the slope $-\frac{\|\mathbf q\|^{2}}{\rho_{0}}$ of the line $-a_{0}(0,\nu)$, then, since both these functions are continuous, it follows, since $\lim_{\nu \to +\infty} f(0,\nu)^{2k}=0$, that there exists a $\nu_{0} > 0$ such that $f(0,\nu)^{2k}+g(0,\nu)^{2k}> -a_{0}(0,\nu)$ for all $\nu \in (0,\nu_{0})$. But the slopes of $f'(0,\nu)^{2k}|_{\nu=0}$ and $g'(0,\nu)^{2k}|_{\nu=0}$ are given by $b_{2}+b_{4}+\ldots b_{2k}$ and $b_{-2}+b_{-4}+\ldots b_{-2k}$. Therefore, we need to show that there exists a $k>0$ such that the inequality 
\begin{equation}\label{ineqb2k}
b_{2}+b_{4}+\ldots b_{2k}+b_{-2}+b_{-4}+\ldots b_{-2k} > -\frac{\|\mathbf q\|^{2}}{\rho_{0}}
\end{equation} 
holds. Recall the formula for $b_{n}$
\[b_{n}= \frac{ \|\mathbf q+n\mathbf p\|^{4}}{\|\mathbf q+n\mathbf p\|^{2}-\|\mathbf p\|^{2}}.\]
The right hand side of the inequality \eqref{ineqb2k} is a constant whereas the left hand side is at least as large as $b_{2k}$ which itself grows like $4k^{2}$. This means that for sufficiently large $k$, the inequality \eqref{ineqb2k} holds which then implies that inequality \eqref{f0g0a0} holds and thus completes the proof.
\end{proof}

\begin{lemma}\label{root}
If $\mathbf q$ is of type $I_{0}$, then there exists a $\nu_{0}>0$ such that if $\nu \in (0,\nu_{0})$ the equation \eqref{fg} has at least one solution $\lambda > 0$.
\end{lemma}
\begin{proof}
The continued fractions $f(\lambda,\nu)$ and $g(\lambda,\nu)$ given, respectively, by equations \eqref{f} and \eqref{g} are continuous functions in $\lambda$. This follows from the Van Vleck and the Stieltjes-Vitali theorems, see \cite[Theorem 4.29 and 4.30]{JT} which guarantee, in particular, that the continued fractions $f$ and $g$ converge and are continuous for all $\lambda > 0$.  
$f(\lambda, \nu)$ and $g(\lambda, \nu)$  are both non-negative and bounded above, respectively, by $1/ a_{1}(\lambda,\nu)$ and $1/a_{-1}(\lambda, \nu)$ and thus satisfy the inequalities $0< f(\lambda, \nu) +g(\lambda, \nu)< \frac{1}{a_{1}(\lambda,\nu)}+\frac{1}{a_{-1}(\lambda,\nu)}$. Since $\lim_{\lambda \to +\infty} \frac{1}{a_{1}(\lambda,\nu)}=\frac{1}{a_{-1}(\lambda,\nu)}=0$ we have that $f(\lambda, \nu) $ and $g(\lambda, \nu)$ go to $0$ as $\lambda \to \infty$. By Lemma \ref{ff0}, $f(\lambda,\nu)+g(\lambda,\nu)$ converges to $f(0,\nu)+g(0,\nu)$ as $\lambda \to 0$.

Since $\mathbf q$ is of type $I_{0}$, $\rho_{0}< 0$ and thus $-a_{0}(\lambda,\nu)=\frac{-\lambda-\nu \|\mathbf q\|^{2}}{\rho_{0}}$ considered as a function of $\lambda$, is a straight line of positive slope whose $y$ intercept is the point $-\frac{\nu\|\mathbf q\|^{2}}{\rho_{0}}$. Lemma \ref{nu0} guarantees the inequality $f(0,\nu)+g(0,\nu) > -a_{0}(0,\nu)= -\frac{\nu\|\mathbf q\|^{2}}{\rho_{0}}$ for all $\nu \in (0,\nu_{0})$ which then means that for all $\nu$ in the interval $(0,\nu_{0})$, the continuous curves $f(\lambda,\nu)+g(\lambda,\nu)$ and $-a_{0}(\lambda,\nu)$ must intersect each other at some positive $\lambda>0$.
\end{proof}

Starting from the existence of a positive root $\lambda >0$ to the continued fractions equation \eqref{fg}, the eigenfunction sequence $(w_{n})$ can be constructed as follows. Define first the continued fractions $u_{n}^{(1)}(\nu,\lambda)$ and $u_{n}^{(2)}(\nu,\lambda)$ using equations \eqref{un1} and \eqref{un2}. Let $u_{n}=u_{n}^{(1)}(\nu,\lambda)$ for $n \geq 0$ and $u_{n}=u_{n}^{(2)}(\nu,\lambda)$ for $n \leq 0$, equality holds at $n=0$ because $\lambda>0$ solves equation \eqref{fg}. Let $(z_{n})$ be the sequence that satisfies $u_{n}=z_{n-1}/z_{n}$ for all $n \in \mathbb Z$. Fixing $z_{0}=1$, a calculation then shows that for $n > 0$, $z_{n}$ satisfies
\begin{equation}\label{zn>0}
z_{n}=\frac{z_{0}}{u_{1}u_{2}\ldots u_{n}}, \quad n > 0,
\end{equation} 
and for each $-n <0$, we have, 
\begin{equation}\label{zn<0}
z_{-n}=z_{0}u_{0}u_{-1}u_{-2}\ldots u_{-n+1}, \quad n > 0.
\end{equation}
$(w_{n})$ is then obtained from the sequence $(z_{n})$ defined above using the formula $w_{n}=z_{n}/\rho_{n}$ for all $n \in \mathbb Z$. The sequence $(w_{n})$ thus constructed is an exponentially decaying sequence, see Lemma \ref{eigvectprop} below. One thus obtains instability of flows of type $I_{0}$ for those positive values $\nu$  in the interval $(0,\nu_{0})$ such that inequality \eqref{ineqf0g0} is satisfied. 

The existence of a solution $\lambda > 0$ to the continued fractions equation \eqref{fg} is thus both sufficient and necessary for $\lambda$ to be an eigenvalue of the operator $L_{B,\bq}$. By the formulae presented in the previous paragraph, the eigensequence $(w_{n})$ is related to the sequence $u_{n}$ which in turn is given by the continued fractions expressions $u_{n}^{(1)}(\lambda, \nu)$ and $u_{n}^{(2)}(\lambda, \nu)$. By construction, $u_{n}^{(1)}(\lambda, \nu)>0$ for $n >0$ and $u_{n}^{(2)}(\lambda, \nu)<0$ for $n \leq 0$. The sequence $(w_n)$ thus possesses certain additional properties. The sign of $w_{n}$ should be such that $u_{n} > 0$ for $n \geq 1$ and $u_{n} < 0$ for $n \leq 0$. Using the formulas $z_{n}=\rho_{n}w_{n}$ and $u_{n}=z_{n-1}/z_{n}$ (and the sign of $\rho_{n}$) one can check directly that the sequence $(w_{n})$ satisfies the following property.
\begin{property}\label{eigvectorions}
If $\mathbf q$ is of type $I_{0}$, the eigenvector $(w_{n})$ of \eqref{evns} is such that the following holds: either $w_{n}>0$ for $n>0$, $w_{n}<0$ for $n=-1,0$, and $(-1)^{|n|}w_{n}>0$ for $n \leq -2$, or the entries of the vector $(-w_{n})$ satisfy the inequalities just listed.
\end{property}
In what follows $\ell_{s}^{2}(\bbZ)$ is the space of square summable sequences with the weight $(1+n^{2s})^{1/2}$.
\begin{lemma}\label{eigvectprop}
In case $\bq$ is of type $I_{0}$, consider the sequence $(w_{n})$ where $w_{n}=z_{n}/\rho_{n}$ and $z_{n}$'s are given by equations \eqref{zn>0} and \eqref{zn<0}. Then the eigenvector sequence $(w_{n})$ is an exponentially decaying sequence and therefore belongs to $\ell_{s}^{2}(\bbZ)$ for any $s \geq 0$.
\end{lemma}
\begin{proof}
It is sufficient to show that $(z_{n})$ is an exponentially decaying sequence since $w_{n}=z_{n}/\rho_{n}$ and $(\rho_{n})$ is a bounded sequence (recall that $\rho_{n} \to 1$ as $n \to \pm \infty$). Consider formulas \eqref{zn>0} and \eqref{zn<0} for $z_{n}$
\begin{equation}\label{zn>0proof}
z_{n}=\frac{z_{0}}{u_{1}u_{2}\ldots u_{n}}, \quad n > 0,
\end{equation} 
\begin{equation}\label{zn<0proof}
z_{-n}=z_{0}u_{0}u_{-1}u_{-2}\ldots u_{-n+1}, \quad n > 0.
\end{equation}
Equations \eqref{u1ineq} and \eqref{u2ineq} in Item (3) of Lemma \ref{ulemma} then guarantee that for all $n \in \mathbb Z$, $z_{n}$ satisfies the following inequality, where $0< q < 1$ and $C$ is a constant
\begin{equation}\label{znineq}
|z_{n}| \leq Cq^{n}. 
\end{equation}
But $Cq^{n}=Ce^{n\log q}=Ce^{-n\delta}$, where $\delta=-\log q > 0$. Thus $z_{n}$ satisfies the inequality
\begin{equation}\label{znineq1}
|z_{n}| \leq Ce^{-n\delta}, \quad \delta >0. 
\end{equation}
Estimate \eqref{znineq1} implies that $\sum_{n \in \mathbb Z}(1+n^{2s})|z_{n}|^{2}$ is summable for all $s \geq 0$ or in other words, the sequence $(z_{n})$ is in $\ell_{s}^{2}(\mathbb Z)$ for all $s \geq 0$.
\end{proof}

\begin{theorem}\label{mainthm}
Suppose \eqref{uni} is a steady state solution to the 2D Navier-Stokes equations \eqref{ns} such that there exists at least one point 
$\mathbf q \in \mathcal Q(\mathbf p)$ of type $I_{0}$, where $\mathbf q$ is not parallel to $\mathbf p$. Also, we assume that $\Gamma \in \mathbb R$ and satisfies the normalization condition $\frac{\Gamma(\mathbf q\wedge\mathbf p)}{\|\bp\|^{2}}=1$. Then there exists a $\nu_{0}>0$ such that for all $\nu$ in the interval $(0,\nu_{0})$, the steady state $\omega^{0}$, $\mathbf u^{0}$ and $f$ defined in 
\eqref{uni} is linearly unstable. In particular, the operator $L_{B,\mathbf q}$ in the space $\ell^{2}_{s}(\mathbb Z)$ has a positive eigenvalue and therefore $L_{B}$ in $\ell_{s}^{2}(\bbZ^{2})$ has a positive eigenvalue. Moreover, the following assertion holds for all $\nu$ in the interval $(0,\nu_{0})$:
if $\mathbf q$ is of type $I_0$ then $\lambda > 0$ is an eigenvalue of $L_{B,\mathbf q}$ with eigenvector $(w_{n})$ satisfying Property \ref{eigvectorions}
if and only if $\lambda > 0$ is a solution to the equation
\begin{equation}\label{eqiov}
a_{0}(\lambda,\nu)+f(\lambda,\nu)+g(\lambda,\nu)=0.
\end{equation}
\end{theorem}
\begin{proof}
Recall formulas for $f(\lambda,\nu)$ and $g(\lambda,\nu)$ given by equations \eqref{f} and \eqref{g}. Suppose $\mathbf q$ is of type $I_{0}$, Lemma \ref{root} guarantees that there exists a $\nu_{0}>0$ so that for all $\nu$ in the interval $(0,\nu_{0})$, equation \eqref{eqiov} has a positive root $\lambda > 0$. Consider the continued fractions $u_{n}^{(1)}(\lambda,\nu)$ for $n \geq 0$ and $u_{n}^{(2)}(\lambda,\nu)$ for $n \leq 0$, given respectively, by equations \eqref{un1} and \eqref{un2}. 
\begin{equation}\label{un1proof}
u_{n}^{(1)}(\lambda,\nu)= a_{n}+[a_{n+1};a_{n+2};\ldots],\, \quad n=0,1,2,\ldots,
\end{equation}
\begin{equation}\label{un2proof}
u_{n}^{(2)}(\lambda,\nu)= -[a_{n-1};a_{n-2};\ldots], \quad n=0,-1,-2,\ldots,
\end{equation}
where, we recall that $a_{n}$ is given by the formula
\begin{equation}\label{anproof}
a_{n}:=a_{n}(\lambda, \nu)= \frac{\lambda + \nu c_{n}}{ \rho_{n}} =\frac{\lambda + \nu \|\mathbf q + n \mathbf p\|^{2}}{\rho_{n}}.
\end{equation}
Notice that $u_{0}^{(1)}(\lambda,\nu)=a_{0}(\lambda,\nu)+f(\lambda,\nu) =-g(\lambda,\nu)=u_{0}^{(2)}(\lambda,\nu)$ because the equation \eqref{eqiov} has a root. Let $u_{n}=u_{n}^{(1)}(\lambda,\nu)$ for $n \geq 0$ and $u_{n}=u_{n}^{(2)}(\lambda,\nu)$ for $n < 0$. Using the continued fraction formulas \eqref{un1proof} and \eqref{un2proof} for $u_{n}$, one can check that the $u_{n}$ satisfies formulas \eqref{unformula1} and \eqref{unformula2}, that is, $u_{n}$ satisfies
\begin{equation}\label{unformulaproof}
u_{n}=a_{n}+\frac{1}{u_{n+1}}, \quad u_{n+1}=\frac{-1}{a_{n}-u_{n}}, \quad n \in \mathbb Z.
\end{equation}
Fix $z_{0}=1$ and consider the expressions for $z_{n}$ given by formulas \eqref{zn>0} and \eqref{zn<0}. 
\begin{equation}\label{zn>0proof}
z_{n}=\frac{z_{0}}{u_{1}u_{2}\ldots u_{n}}, \quad n > 0,
\end{equation} 
\begin{equation}\label{zn<0proof}
z_{-n}=z_{0}u_{0}u_{-1}u_{-2}\ldots u_{-n+1}, \quad n > 0.
\end{equation}
$z_{n}$ thus defined satisfies $u_{n}=z_{n-1}/z_{n}$ for every $n \in \mathbb Z$. Plugging this in \eqref{unformulaproof} we get that the sequence $z_{n}$ satisfies equation \eqref{znformula}
\begin{equation}\label{znformulaproof}
z_{n-1}-z_{n+1}= a_{n}z_{n}.
\end{equation}
If we now let $w_{n}=z_{n}/\rho_{n}$ for all $n \in \mathbb Z$, then one can check that the $w_{n}$ so constructed satisfies the eigenvalue equation \eqref{evns} 
\begin{equation}\label{evnsproof}
\rho_{n-1}w_{n-1}-\rho_{n+1}w_{n+1}=\lambda w_{n}+ \nu \|\mathbf q + n \mathbf p\|^{2}w_{n}.
\end{equation}
Using the fact that $u_{n}>0$ for $n >0$ and $u_{n}<0$ for $n \leq 0$ and the formulas for $z_{n}$ and $w_{n}$ constructed above, one can directly verify that $(w_{n})$ satisfies the properties listed in Property \ref{eigvectorions}. The fact that $(w_{n})$ is in the space $\ell^{2}_{s}(\mathbb Z)$ for all $s \geq 0$ is due to Lemma \ref{eigvectprop}.

Now, suppose $\lambda > 0$ is an eigenvalue of $L_{B,\mathbf q}$ with eigenvector $(w_{n})$ satisfying property \ref{eigvectorions}. Starting with the eigenvalue equation \eqref{evnsproof}, let $z_{n}=\rho_{n}w_{n}$, (recall formula \eqref{anproof}) to obtain equation \eqref{znformulaproof}. Notice that by property \ref{eigvectorions}, $w_{n} \neq 0$ for any $n \in \mathbb Z$ and thus $z_{n} \neq 0$ for any $n \in \mathbb Z$. Now if we let $u_{n}=z_{n-1}/z_{n}$, we see that $u_{n}$ satisfies \eqref{unformulaproof}. Iterating the first formula in \eqref{unformulaproof} forwards for $n \geq 0$ and the second formula in \eqref{unformulaproof} backwards for $n \leq 0$, we obtain expressions for the continued fractions $u_{n}^{(1)}(\lambda,\nu)$ and $u_{n}^{(2)}(\lambda,\nu)$ given, respectively, by equations \eqref{un1proof} and \eqref{un2proof}. These continued fractions match at $n=0$, i.e., we have $u_{0}^{(1)}(\lambda,\nu)=u_{0}^{(1)}(\lambda,\nu)$ for the given $\lambda > 0$. Since $u_{0}^{(1)}(\lambda,\nu)=a_{0}(\lambda,\nu)+f(\lambda,\nu)$ and $u_{0}^{(2)}(\lambda,\nu)=-g(\lambda,\nu)$, we see that \eqref{eqiov} has a positive root $\lambda > 0$.
\end{proof}
See Figure \ref{ENS} for a numerical illustration to our main theorem. The figure corresponds to $\mathbf p =(3,1)$ and $\mathbf q=(-1,2)$. The left panel is the Navier-Stokes case with viscosity $\nu=0.06$, the right panel is the Euler ($\nu=0$) case for comparison. The respective continued fractions have been approximated by their $10$-th truncations. We also wish to mention that a website to numerically compute and display roots of the equation \eqref{eqiov} for the Navier-Stokes (together with related equations for the Euler equations and the $\alpha$-Euler model) is currently under development by Aleksei Seletskiy, a high school student in Palo Alto, CA, U.\ S.\ A. This will eventually be available at https://thetazero.github.io/capstone/index.    
\begin{figure}[H]\label{ENS}
\begin{center}
\includegraphics[scale=.5,keepaspectratio,draft=false]{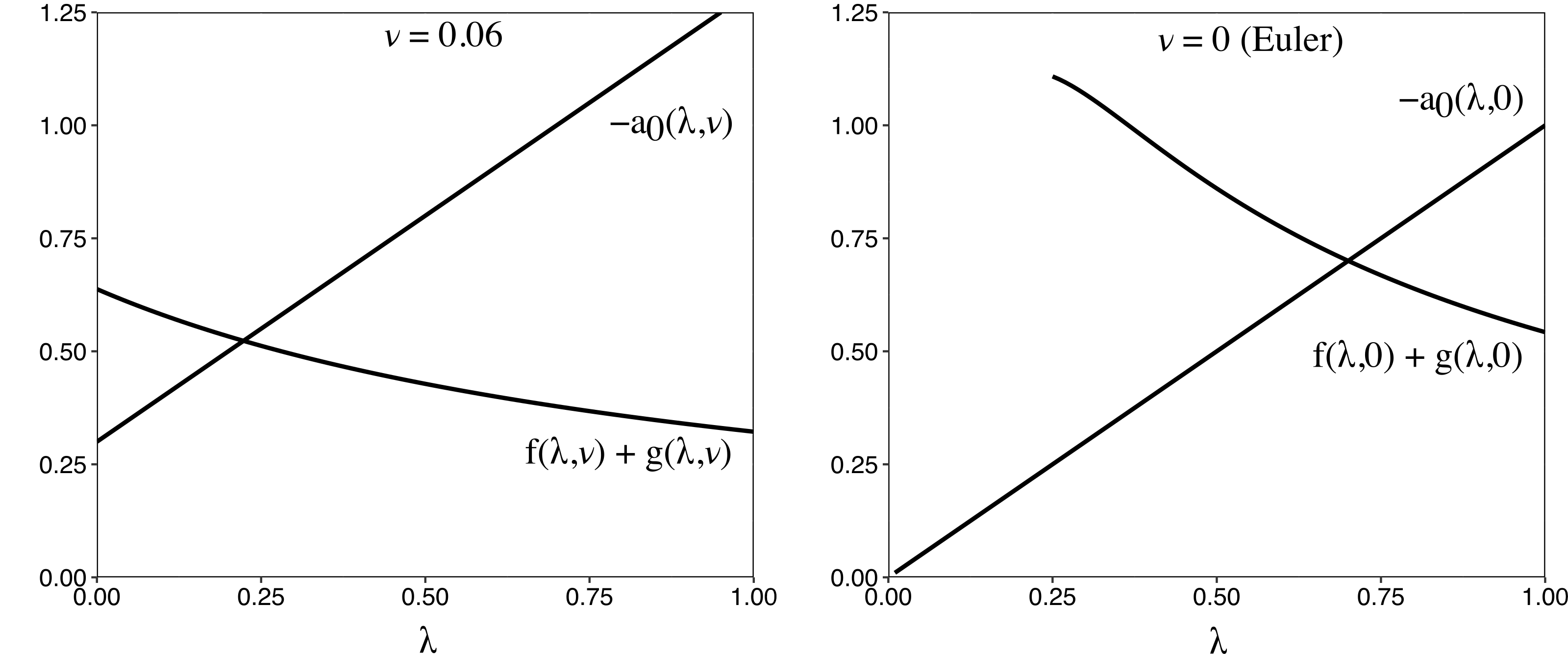}
\caption{Numerical illustration of roots of equation \ref{eqiov}. Here $\mathbf p=(3,1)$ and $\mathbf q=(-1,2)$. The left panel is the Navier-Stokes case with viscosity $\nu=0.06$, the right panel is the Euler ($\nu=0$) case for comparison. The respective continued fractions have been approximated by their $10$-th truncations.}
\end{center}
\end{figure}

\subsection{Modifications in case $\mathbf q$ is of type $I_{+}$ and $I_{-}$}
We outline changes to the results above in case $\mathbf q$ is of type $I_{+}$ and $I_{-}$. By Remark \ref{rho1}, if $\mathbf q$ is of type $I_{+}$ then $\rho_{1}=0$ and if $\mathbf q$ is of type $I_{-}$ then $\rho_{-1}=0$. Consider the eigenvalue problem \eqref{evns}
\begin{equation}\label{evnsip}
\rho_{n-1}w_{n-1}-\rho_{n+1}w_{n+1}=\lambda w_{n}+ \nu \|\mathbf q + n \mathbf p\|^{2}w_{n}.
\end{equation}
Recall formula $a_{n}:=a_{n}(\lambda, \nu)= \frac{\lambda + \nu \|\mathbf q + n \mathbf p\|^{2}}{\rho_{n}}$. $a_{1}$ and $a_{-1}$ are undefined, respectively, for $\mathbf q$ of type $I_{+}$ and $I_{-}$. 

Consider first the case $I_{+}$. Letting $z_{n}=\rho_{n}w_{n}$, note that $z_{1}=0$, reduces equation \eqref{evnsip} to
\begin{equation}\label{znip}
z_{n-1}-z_{n+1}=a_{n}(\lambda,\nu)z_{n}, \quad n \neq 1.
\end{equation}
When $n \leq 1$ and when $n \geq 3$, we have the equation \eqref{znip} above. When $n=0$, we have
\begin{equation}\label{z-10}
z_{-1}=a_{0}z_{0},
\end{equation}
and when $n=2$ we have
\begin{equation}\label{z-23}
-z_{3}=a_{2}z_{2}.
\end{equation}
Notice that equations \eqref{znip} for $n \leq 0$ and $n \geq 2$ are uncoupled.
Assume that $z_{n} \neq 0$ for $n \neq 1$ and setting $u_{n}=z_{n-1}/z_{n}$ for $n \neq 1$ we first note that $u_{2}=z_{1}/z_{2}=0$. Also, from \eqref{z-10} and \eqref{z-23}, we obtain that $u_{0}=a_{0}$ and $u_{3}=-1/a_{2}$. Equation \eqref{znip} thus becomes two separate equations in the variable $u_{n}$ with initial conditions given by
\begin{equation}\label{unformula1ip1}
u_{n}=a_{n}+\frac{1}{u_{n+1}}, n \leq -1, \quad u_{0}=a_{0},
\end{equation}
and
\begin{equation}\label{unformula1ip2}
u_{n}=a_{n}+\frac{1}{u_{n+1}}, n \geq 3, \quad u_{3}=-1/a_{2}.
\end{equation}
Rewriting \eqref{unformula1ip1} as $u_{n+1}=-1/(a_{n}-u_{n})$ and iterating backwards, for $n=0$, we see that $u_{0}=a_{0}$ must be given as the continued fraction
\begin{equation}\label{ipeq}
u_{0}=a_{0}(\lambda,\nu)=u_{0}^{(2)}(\lambda,\nu)=-[a_{-1};a_{-2};\ldots],
\end{equation}
where the formula for the continued fraction $u_{0}^{(2)}(\lambda,\nu)$ is the same as before, see equation \eqref{un12}. %
If we rewrite equation \eqref{unformula1ip2} as $u_{n+1}=-1/(a_{n}-u_{n})$ and solve with the initial condition $u_{3}=-1/a_{2}$ we run into the following problem. $|u_{n+1}|$ is bounded above by $1/a_{n}$ and since $a_{n} \to \infty$ as $n \to \infty$, we have that $\lim_{n \to \infty} u_{n+1}=0$. If we attempt to reconstruct the sequence $z_{n}$ from the $u_{n}$ thus obtained we will have the formula
\begin{equation}
z_{n}=\frac{z_{2}}{u_{3}u_{4}\ldots u_{n}},
\end{equation}
where $z_{2}$ is arbitrary. The fact that $\lim_{n \to \infty} u_{n}  = 0$ implies 
that the $z_{n}$ thus constructed grows exponentially. There is potentially another problem with this construction. Since $u_{n}=a_{n}+1/u_{n+1}$ for $n \geq 3$, iterating this forward, we see that $u_{3}$ is given by the continued fraction expression $a_{3}+[a_{4};a_{5}\ldots]$ (recall continued fraction Notation \ref{cf-notation}). The continued fraction thus constructed is positive, since $a_{n}$, $n \geq 3$ are all positive and thus cannot match the initial condition $u_{3}=-1/a_{2}$.  

These problems arise because $z_{2}$ was chosen to be non-zero. We avoid these problems by setting $z_{2}=0$. Then $z_{3}$ becomes $0$ by \eqref{z-23} and all the succeeding $z_{n}$ are also zero using \eqref{znip}. This then means that $w_{n}=0$ for $n >1$. We find a root of equation \eqref{ipeq} and set $z_{n}$ to be zero for all $n \geq 1$. 

Summarizing the discussion above, we obtain, in the case $\mathbf q$ is of type $I_{+}$, $\lambda>0$ is an eigenvalue of equation \eqref{evnsip} if and only if $\lambda > 0$ is a root of the equation \eqref{ipeq} and moreover, the eigenvector $(w_{n})$ satisfies the property that $w_{n}=0$ for all $n > 1$. We only consider the continued fraction $u_{n}^{(2)}(\lambda,\nu)$ for $n \leq 0$ given by \eqref{un2} and the continued fraction $g(\lambda,\nu)$ given by \eqref{g} and we wish to find a root $\lambda >0$ of the equation
\begin{equation}\label{eqa0gip}
a_{0}(\lambda,\nu)+g(\lambda,\nu)=0.
\end{equation}

Let us now consider the case $I_{-}$. As before, let $z_{n}=\rho_{n}w_{n}$, note that $z_{-1}=0$, reduces the eigenvalue equation \eqref{evnsip} to
\begin{equation}\label{znim}
z_{n-1}-z_{n+1}=a_{n}(\lambda,\nu)z_{n}, \quad n \neq -1.
\end{equation}
When $n \leq -3$ and when $n \geq 1$, we have the equation \eqref{znim} above. When $n=-2$, we have
\begin{equation}\label{z-32}
z_{-3}=a_{-2}z_{-2},
\end{equation}
and when $n=0$ we have
\begin{equation}\label{z01}
z_{1}=-a_{0}z_{0}.
\end{equation}
Similar to the reasons described for $I_{+}$, we will need to choose $z_{-2}$ to be $0$ (for otherwise the $z_{n}$ thus constructed for $n \leq -2$ will be exponentially unbounded). Equation \eqref{z-32} implies $z_{-3}=0$ and all the preceeding $z_{n}$ are zero for $n \leq -3$.
Consider \eqref{znim} for $n \geq 0$. Assuming $z_{n} \neq 0$ for $n \geq 1$ and setting $u_{n}=z_{n-1}/z_{n}$, we first note that $u_{0}=z_{-1}/z_{0}=0$ and  $u_{1}=z_{0}/z_{1}=-1/a_{0}$. Equation \eqref{znim} then reduces to
\begin{equation}\label{unformula1im2}
u_{n}=a_{n}+\frac{1}{u_{n+1}},  n \geq 0, \quad u_{0}=0.
\end{equation}
Iterating this forward, we see that $u_{0}=0$ must match the continued fraction $a_{0}+[a_{1};a_{2};\ldots]$ which is given by (recall formula \eqref{un12}) $u_{0}^{(1)}(\lambda,\nu)$. That is, $u_{0}$ is given by
\begin{equation}\label{imeq}
0=u_{0} = u_{0}^{(1)}(\lambda,\nu)=a_{0}(\lambda,\nu)+[a_{1};a_{2};\ldots].
\end{equation}
Recalling formula \eqref{f}, we obtain, in the case $\mathbf q$ is of type $I_{-}$, that $\lambda > 0$ is an eigenvalue of the equation \eqref{znim} if and only if $\lambda > 0$ solves equation \eqref{imeq} which can be rewritten as 
\begin{equation}\label{eqa0fim}
a_{0}(\lambda,\nu)+f(\lambda,\nu)=0.
\end{equation}
Moreover, the eigenvector $(w_{n})$ satisfies the property that $w_{n}=0$ for all $n < -1$. 
We only consider the continued fraction $u_{n}^{(1)}(\lambda,\nu)$ for $n \geq 0$ given by \eqref{un1} and the continued fraction $f(\lambda,\nu)$ given by \eqref{g} and we wish to find a root $\lambda >0$ of the equation
\begin{equation}\label{eqa0fim2}
a_{0}(\lambda,\nu)+f(\lambda,\nu)=0.
\end{equation}
Equations \eqref{eqa0gip} and \eqref{eqa0fim2} have positive roots provided the following inequalities hold
\begin{equation}\label{ineqg0ip}
g(0,\nu) > -a_{0}(0,\nu)= -\frac{\nu\|\mathbf q\|^{2}}{\rho_{0}}
\end{equation}
and 
\begin{equation}\label{ineqf0im}
f(0,\nu) > -a_{0}(0,\nu)= -\frac{\nu\|\mathbf q\|^{2}}{\rho_{0}}.
\end{equation}
Lemma \ref{nu0} now reads
\begin{lemma}\label{nu0ipim}
There exist $\nu_{0}^{p}>0$ and $\nu_{0}^{m}>0$ such that for all $\nu$ in the interval $(0,\nu_{0}^{p})$ and $(0,\nu_{0}^{m})$ respectively the inequalities in equations \eqref{ineqg0ip} and \eqref{ineqf0im} are satisfied.
\end{lemma}
\begin{proof}
The proof is similar to the proof of Lemma \ref{nu0} except that instead of using the properties of both the even truncated continued fractions for $f(0,\nu)^{2k}$ and $g(0,\nu)^{2k}$, we instead use separately, the corresponding properties for $g(0,\nu)^{2k}$ from Lemma \ref{even-trunc-properties} to prove inequality \eqref{ineqg0ip} and the properties for $f(0,\nu)^{2k}$ to prove inequality \eqref{ineqf0im}.
\end{proof}
Lemma \ref{root} now reads
\begin{lemma}\label{rootipim}
If $\mathbf q$ is, respectively, of type $I_{+}$ and type $I_{-}$, then there exist a $\nu_{0}^{p}>0$ and $\nu_{0}^{m}>0$ such that if $\nu$ is respectively  in the interval $(0,\nu_{0}^{p})$ and $(0,\nu_{0}^{m})$ the equations \eqref{eqa0gip} and \eqref{eqa0fim} have at least one solution $\lambda > 0$.
\end{lemma}
Property \ref{eigvectorions} is modified as follows:
\begin{property}\label{eigvectorionsipim}
\noindent
\begin{enumerate}

\item[(1)] In case $\bq$ is of type $I_{+}$, the eigenvector $(w_{n})$ of \eqref{evnsip} is such that the following holds: either $w_{n}=0$ for $n>1$, $w_{1}>0$, $w_{n}<0$ for $n=-1,0$, and $(-1)^{|n|}w_{n}>0$ for $n \leq -2$, or the entries of the vector $(-w_{n})$ satisfy the inequalities just listed. 
\item[(2)] In case $\bq$ is of type $I_{-}$, the eigenvector $(w_{n})$ of \eqref{evnsip} is such that the following holds: either $w_{n}=0$ for $n< -1$, $w_{n}<0$ for $n=-1,0$, and $w_{n}>0$ for $n > 0$, or the entries of the vector $(-w_{n})$ satisfy the inequalities just listed.
\end{enumerate}
\end{property}
The changes to the statement of Theorem \ref{mainthm} are as follows. 
\begin{theorem}\label{mainthmipim}
Suppose \eqref{uni} is a steady state solution to the 2D Navier-Stokes equations \eqref{ns} such that there exists at least one point 
$\mathbf q \in \mathcal Q(\mathbf p)$ respectively of type $I_{+}$ or type $I_{-}$, where $\mathbf q$ is not parallel to $\mathbf p$. Also, we assume that $\Gamma \in \mathbb R$ and satisfies the normalization condition $\frac{\Gamma(\bq\wedge\bp)}{2\|\bp\|^{2}}=1$. Then there exist a $\nu_{0}^{p}>0$ and a $\nu_{0}^{m}>0$ such that for all $\nu$ in the respective intervals $(0,\nu_{0}^{p})$ and $(0,\nu_{0}^{m})$, the steady state $\omega^{0}$, $\mathbf u^{0}$ and $f$ defined in 
\eqref{uni} are linearly unstable respectively when $\mathbf q$ is of type $I_{+}$ and $I_{-}$. In particular, the operator $L_{B,\mathbf q}$ in the space $\ell^{2}_{s}(\mathbb Z)$ has a positive eigenvalue and therefore $L_{B}$ in $\ell_{s}^{2}(\bbZ^{2})$ has a positive eigenvalue. Moreover, the following assertions hold for all $\nu$ in the respective intervals $(0,\nu_{0}^{p})$ and $(0,\nu_{0}^{m})$:
\begin{enumerate}
\item[(1)] If $\mathbf q$ is of type $I_+$ then $\lambda > 0$ is an eigenvalue of $L_{B,\mathbf q}$ with eigenvector $(w_{n})$ satisfying Property \ref{eigvectorionsipim} (1) if and only if $\lambda > 0$ is a solution to the equation
\begin{equation}\label{eqiovip}
a_{0}(\lambda,\nu)+g(\lambda,\nu)=0.
\end{equation}
\item[(2)] If $\mathbf q$ is of type $I_-$ then $\lambda > 0$ is an eigenvalue of $L_{B,\mathbf q}$ with eigenvector $(w_{n})$ satisfying Property \ref{eigvectorionsipim} (2) if and only if $\lambda > 0$ is a solution to the equation
\begin{equation}\label{eqiovim}
a_{0}(\lambda,\nu)+f(\lambda,\nu)=0.
\end{equation}
\end{enumerate}
\end{theorem}
\section{A Fredholm determinant characterization of instability}
In this section, we offer a characterization of unstable eigenvalues $\lambda$ of the operator $L_{B,\mathbf q}$ in terms of Fredholm determinants. Let us consider the eigenvalue equation \eqref{evns} which can be rewritten as
\begin{equation}
L_{B, \mathbf q} (w_{n})=((S-S^*)\diag{\rho_{n}}-\nu \|\mathbf q+n \mathbf p\|^{2})(w_{n})=\lambda(w_{n}),
\end{equation}
where $S:(w_{n}) \mapsto (w_{n-1})$ denotes the shift operator on $\ell^{2}(\mathbb Z)$ and $S^{*}$ denotes its adjoint. Suppose $\lambda > 0$ is in the point spectrum $\sigma_{p}(L_{B,\mathbf q})$. The operator of multiplication by $-\nu \|\mathbf q+n \mathbf p\|^{2}$ has spectrum on the negative real axis, so if $\lambda>0$ then $\lambda$ is in the resolvent set of the operator $-\nu \|\mathbf q+n \mathbf p\|^{2}$. Consider the following factorization
\begin{align*}
L_{B, \mathbf q}-\lambda &= (S-S^*)\diag{\rho_{n}}-\nu \|\mathbf q+n \mathbf p\|^{2}-\lambda \\
&=(-\nu \|\mathbf q+n \mathbf p\|^{2}-\lambda)\big(I+  (-\nu \|\mathbf q+n \mathbf p\|^{2}-\lambda)^{-1}(S-S^*)\diag{\rho_{n}}\big)\\
&=(-\nu \|\mathbf q+n \mathbf p\|^{2}-\lambda)\big(I+K_{\lambda}\big) \numberthis \label{factorization},
\end{align*}
where we have denoted by $K_{\lambda}$ the operator
\begin{equation}
K_{\lambda}=(-\nu \|\mathbf q+n \mathbf p\|^{2}-\lambda)^{-1}(S-S^*)\diag\{\rho_{n}\}.
\end{equation}
Notice that $(S-S^*)$ and $\diag{\rho_{n}}$ are bounded operators, the operator $(-\nu \|\mathbf q+n \mathbf p\|^{2}-\lambda)^{-1}$ is trace class since 
$\sum_{n \in \mathbb Z} (-\nu \|\mathbf q+n \mathbf p\|^{2}-\lambda)^{-1} < +\infty$. Thus $K_{\lambda}$ being a product of bounded operators and trace class operator is itself trace class. One can thus attach a perturbation determinant to it, (see \cite[Chapter 7; Theorem 6.1]{GGK90} for example, for a construction of this determinant), i.e., we have
\begin{equation}\label{mathcald}
\mathcal D(\lambda)=\det(I+K_{\lambda})=\prod_{n}(1+\kappa_\lambda^{n}),
\end{equation}
where the index $n$ in the above product ranges over all the eigenvalues $\kappa_\lambda^{n}$ of the operator $K_{\lambda}$. We thus have
\begin{theorem}\label{fredchar}
$\lambda>0$ is an eigenvalue of the operator $L_{B,\mathbf q}$ if and only if $\det(I+K_{\lambda})=0$ if and only if $-1 \in \sigma_{p}(K_{\lambda})$.
\end{theorem}
\begin{proof}
Since $\lambda$ is in the resolvent set of the operator $-\nu \|\mathbf q+n \mathbf p\|^{2}$, the factorization in \eqref{factorization} implies that $\lambda >0$ is in the spectrum of $L_{B,\mathbf q}$ if and only if $-1$ is in the spectrum of $K_{\lambda}$. The fact that $-1$ is in the spectrum of $K_{\lambda}$ if and only if $\det(I+K_{\lambda})=0$ follows from the product formula \eqref{mathcald} above. 
\end{proof}
See \cite[Theorem 4.2]{LV18} for a related characterization of the unstable eigenvalues of the linearized vorticity operator $L_{\rm vor}$ for the 2D Euler equations in terms of 2-modified perturbation determinants associated with a Hilbert-Schmidt operator. %
Note that in theorem \ref{fredchar} above, there is no reference to whether $\mathbf q$ is of type $I$. If, in addition, we specialize to the case that $\mathbf q$ is of type $I$, we have a one to one correspondence between the roots of equation \eqref{fg} and zeros of the perturbation determinant $\det(I+K_{\lambda})=0$. Every root of the equation \eqref{fg} contributes a zero to the perturbation determinant $\det(I+K_{\lambda})=0$.
\begin{theorem}\label{1-1}
Suppose $\lambda > 0$ is a root of the equation \eqref{fg}, then $\det(I+K_{\lambda})=0$. In addition, 
if $\lambda>0$ is an eigenvalue of the operator $L_{B,\mathbf q}$ with eigenvector $(w_{n})$ satisfying the property , then $\lambda$ is a root of \eqref{fg} if and only if $\det(I+K_{\lambda})=0$.
\end{theorem}
\begin{proof}
By Theorem \ref{mainthm} $\lambda > 0$ is in the spectrum of $L_{B,\mathbf q}$ with eigenvector $(w_{n})$ satisfying Property \ref{eigvectorions} if and only if $\lambda$ is a root of the equation \eqref{fg}. By Theorem \ref{fredchar} $\lambda > 0$ is in the spectrum of $L_{B,\mathbf q}$ if and only if $\det(I+K_{\lambda})=0$. These two facts imply the result.
\end{proof}
Letting $k_{n}=k_{n}(\lambda)=\frac{1}{(-\nu \|\mathbf q+n \mathbf p\|^{2}-\lambda)}$, we notice that
$K_{\lambda}$ is represented by the following matrix
\begin{equation}
 K_{\lambda}=
\left( \begin{array}{ccccc}  \ddots & & & &  \\
 &0& -\rho_{0}k_{-1}&0& \\
 & k_{0}\rho_{-1} & \text{\fbox{$0$}} & -\rho_{1}k_{0} &\\
 & 0 & k_{1}\rho_{0} & 0 &  \\ 
& & & &  \ddots  
\end{array}\right).
\end{equation}
That is, $K_{\lambda}$ is a bidiagonal operator with non-zero entries only above and below the main diagonal. Here, $k_{n}=O(n^{-2})$ as $n \to \pm \infty$. The objective is to show that $-1 \in \sigma_{p}(K_{\lambda})$. In other words, find a $\lambda$ (possibly complex), where the real part of $\lambda > 0$ such that $-1$ is in the spectrum of $K_{\lambda}$. Instead of studying the roots of the continued fraction equation \eqref{fg}, an alternative way to study the spectrum of the operator $L_{B,\mathbf q}$ is to study the spectrum of the operator $K_{\lambda}$. 
\section{Extensions to the second grade fluid model, Navier-Stokes-$\alpha$ and Navier-Stokes-Voigt models}
In this section, we shall consider extensions of the main instability theorem \ref{mainthm} to regularized fluid models, the second grade fluid model, Navier-Stokes-$\alpha$ and the Navier-Stokes-Voigt models. Fix $\alpha>0$ and let the filtered / smoothed / regularized velocity $\mathbf u_{f}$ satisfy
\begin{equation}\label{ufdef} 
\mathbf u=(I-\alpha^{2}\Delta)\mathbf u_{f}, 
\end{equation}
where $\mathbf u$ is the velocity field. On appropriate function spaces, assuming that the Helmholtz operator $I-\alpha^{2}\Delta$ is invertible, $\mathbf u_{f}=(I-\alpha^{2}\Delta)^{-1}\mathbf u$ is smoother than the actual velocity $\mathbf u$ and is referred to as the filtered (or regularized or smoothed) velocity. We consider all of the following models on the two torus $\mathbb T^{2}$. The second grade fluid model, see \cite{LNTZ15a} and references therein, is given by the equations
\begin{align}\label{secondgradevel}
\partial_{t}(I-\alpha^{2}\Delta)\mathbf u_f + \mathbf u_f \cdot \nabla (I-\alpha^{2}\Delta)\mathbf u_f + &(\nabla \mathbf u_f)^{T}\cdot (I-\alpha^{2}\Delta)\mathbf u_f\nonumber \\&=-\nabla p+\nu \Delta \mathbf u_f+\nu f_{1},\\
\div \cdot \mathbf u_f&=0. \nonumber
\end{align}
In the above, $(\cdot)^{T}$ denotes the transpose of the matrix $(\cdot)$. The Navier-Stokes-$\alpha$, see \cite{FHT01}, is given by the equations
\begin{align}\label{nsalphavel}
\partial_{t}(I-\alpha^{2}\Delta)\mathbf u_f + \mathbf u_f \cdot \nabla (I-\alpha^{2}\Delta)\mathbf u_f + &(\nabla \mathbf u_f)^{T}\cdot (I-\alpha^{2}\Delta)\mathbf u_f\nonumber \\&=-\nabla p+\nu \Delta (I-\alpha^{2}\Delta)\mathbf u_f+\nu f_{1},\\
\div \cdot \mathbf u_f&=0. \nonumber
\end{align}
The only difference between the second grade fluid model and the Navier-Stokes-$\alpha$ model is in the viscosity term: $\nu \Delta \mathbf u_{f}$ for the second grade model versus $\nu \Delta (1-\alpha^{2}\Delta)\mathbf u_{f}$ for the Navier-Stokes-$\alpha$ model. In both these models, formally putting $\alpha=0$ and $\nu=0$ results in the 2D Euler equations. If we set $\alpha=0$ we obtain the Navier-Stokes equations and if we set $\nu=0$ we obtain the 2D $\alpha$-Euler equations. 
The Navier-Stokes-Voigt model, see for example \cite{BB12}, is given by the equations
\begin{align}\label{nsvoigtvel}
\partial_{t}(I-\alpha^{2}\Delta)\mathbf u_{f} + \mathbf u_{f} \cdot \nabla \mathbf u_{f} + &=-\nabla p+\nu \Delta \mathbf u_{f}+\nu f_{1},\\
\div \cdot \mathbf u_{f}&=0. \nonumber
\end{align}
\subsection{Second grade fluid model}
We first study the second grade fluid model. Taking $\curl$ of equation \eqref{secondgradevel} and setting 
\begin{equation}\label{vortformulasecondgrade}
\omega=\curl(I-\alpha^{2}\Delta)\mathbf u_{f},
\end{equation}
we obtain the equation
\begin{equation}\label{secondgradevort}
\partial_{t}\omega+\mathbf u_{f}\cdot \nabla \omega=\nu \Delta (I-\alpha^{2}\Delta)^{-1}\omega+\nu f.
\end{equation}
Denote by $\varphi$, the stream function corresponding to $\mathbf u_{f}$, i.e., $\mathbf u_{f}=(\varphi_{y},-\varphi_{x})$. Note first that $\curl \mathbf u_{f}=\curl (\varphi_{y},-\varphi_{x})=-\Delta \varphi$. We thus have that $\omega=\curl(I-\alpha^{2}\Delta)\mathbf u_{f}=-\Delta(I-\alpha^{2}\Delta)\varphi$.
Fix $\mathbf p \in \mathbb Z^{2}$, where $\mathbf p=(p_{1},p_{2})$ and denote $\mathbf p^{\perp}=(-p_{2},p_{1})$. Also, fix $\Gamma \in \mathbb R$ and  consider a steady state to \eqref{secondgradevort} of the form
\begin{align}\label{steadystatesecondgrade}
\omega^{0}=\Gamma \cos (\mathbf p \cdot \mathbf x), \quad \mathbf u_{f}^{0}=\frac{\Gamma \sin (\mathbf p \cdot \mathbf x) \mathbf p^{\perp}}{\|\mathbf p\|^{2}(1+\alpha^{2}\|\mathbf p\|^{2})}, \nonumber \\
f=-(I-\alpha^{2}\Delta)^{-1}\Delta \omega^{0}=\frac{\|\mathbf p\|^{2}\Gamma \cos (\mathbf p \cdot \mathbf x)}{1+\alpha^{2}\|\mathbf p\|^{2}}.
\end{align}
The stream function $\varphi^{0}$ that corresponds to $\omega^{0}$ is given by the formula \[\varphi^{0}=\frac{\Gamma \cos (\mathbf p \cdot \mathbf x)}{\|\mathbf p\|^{2}(1+\alpha^{2}\|\mathbf p\|^{2})}.\]
Using the Fourier series decomposition  $\omega(\mathbf x)=\sum_{\mathbf k\in\mathbb Z^2\setminus\{0\}}\omega_\mathbf k e^{i \mathbf k\cdot\mathbf x}$  
we rewrite \eqref{secondgradevort} as
\begin{equation}\label{sgfourier}
\frac{d\omega_{\mathbf k}}{dt}=\sum_{\mathbf q\in\mathbb Z^2\setminus\{0\}}\beta_{1}(\mathbf k-\mathbf q,\mathbf q)\omega_{\mathbf k-\mathbf q}\omega_\mathbf q - \nu \frac{\|\mathbf k\|^{2}}{(1+\alpha^{2}\|\mathbf k\|^{2})}\omega_{\mathbf k} + \nu f_{\mathbf k} ,\,\,\mathbf k\in\mathbb Z^2\setminus\{0\},
\end{equation}
where the coefficients $\beta_{1}(\mathbf p,\mathbf q)$ for $\mathbf p,\mathbf q\in\mathbb Z^2$ are defined as 
\begin{equation}\label{dfnsg}
\beta_{1}(\mathbf p,\mathbf q)=\frac{1}{2}\bigg(\|\mathbf q\|^{-2}(1+\alpha^{2}\|\mathbf q\|^{2})^{-1}-\|\mathbf p\|^{-2}(1+\alpha^{2}\|\mathbf p\|^{2})^{-1}\bigg)(\mathbf p\wedge\mathbf q)\,
\end{equation}
for 
$\mathbf p\neq0,\mathbf q\neq0$, and $\beta_{1}(\mathbf p,\mathbf q)=0$ otherwise.

Consider the linearization of \eqref{secondgradevort} about the steady state \eqref{steadystatesecondgrade} given by
\begin{equation}\label{sglin}
\partial_{t}\omega + \mathbf u_{f}^{0} \cdot \nabla \omega + \mathbf u_{f} \cdot \nabla \omega^{0}= \nu \Delta (I-\alpha^{2}\Delta)^{-1} \omega.
\end{equation} 
Corresponding to \eqref{sglin}, consider the linear operator $L_{B}^{\alpha}$ given by
\begin{equation}
L_{B}^{\alpha} \omega = - \mathbf u_{f}^{0} \cdot \nabla \omega - \mathbf u_{f} \cdot \nabla \omega^{0}+ \nu \Delta (1-\alpha^{2}\Delta)^{-1} \omega.
\end{equation}
Using the Fourier decomposition in \eqref{sgfourier}, consider the following linearized vorticity operator in the space $\ell^2(\mathbb Z^2)$
\begin{align}
L_B^{\alpha}&: (\omega_\mathbf k)_{\mathbf k\in\mathbb Z^2}\mapsto\nonumber\\&
\big(\beta_{1}(\mathbf p,\mathbf k-\mathbf p)\Gamma\omega_{\mathbf k-\mathbf p}-
\beta_{1}(\mathbf p,\mathbf k+\mathbf p)\Gamma \omega_{\mathbf k+\mathbf p}- \nu \|\mathbf k\|^{2}(1+\alpha^{2}\|\mathbf k\|^{2})^{-1}\omega_{\mathbf k}\big)_{\mathbf k\in\mathbb Z^2}.\label{dfnLBa}
\end{align}
Fixing a $\mathbf q$ of type $I_{0}$ and decomposing spaces and operators as before, consider the operator $L_{B,\mathbf q}^{\alpha}$ acting on the space $\ell_{2}(\mathbb Z)$ given by
\begin{align}
L_{B,\mathbf q}^{\alpha}: (w_n)_{n\in\mathbb Z}\mapsto
\bigg( \beta_{1}&(\mathbf p,\mathbf q+(n-1)\mathbf p)\Gamma w_{n-1}-
\beta_{1}(\mathbf p,\mathbf q+(n+1)\mathbf p)\Gamma w_{n+1} \nonumber \\&- \nu (1+\alpha^{2}\|\mathbf q + n \mathbf p \|^{2})^{-1}\|\mathbf q + n \mathbf p \|^{2} w_{n} \bigg)_{n\in\mathbb Z}.\label{sgLBq}
\end{align}
Let $\rho_{n}'$ be defined by the equation
\begin{equation}\label{rhonsg}
\rho_{n}'=\Gamma \beta_{1}(\mathbf p,\mathbf q+n\mathbf p)=\frac{(\mathbf q \wedge \mathbf p) \Gamma}{2 (1+\alpha^{2}\|\mathbf p\|^{2})\|\mathbf p\|^{2}}\bigg( 1-\frac{(1+\alpha^{2}\|\mathbf p\|^{2})\|\mathbf p\|^{2}}{(1+\alpha^{2}\|\mathbf q+n \mathbf p\|^{2})\|\mathbf q+n \mathbf p\|^{2}}\bigg).
\end{equation}
Normalize $\Gamma$ so that 
\begin{equation}\label{gnormsg}
\frac{(\mathbf q \wedge \mathbf p) \Gamma}{2 (1+\alpha^{2}\|\mathbf p\|^{2})\|\mathbf p\|^{2}}=1.
\end{equation}
After this normalization, $\rho_{n}' \to 1$ as $n \to \pm \infty$. Using \eqref{rhonsg} the operator $L_{B,\mathbf q}^{\alpha}$ can be rewritten as 
\begin{equation}\label{LBq-sg2}
L_{B,\mathbf q}^{\alpha}: (w_n)_{n\in\mathbb Z}\mapsto (\rho_{n-1}'w_{n-1}-\rho_{n+1}'w_{n+1}- \nu (1+\alpha^{2}\|\mathbf q + n \mathbf p \|^{2})^{-1}\|\mathbf q + n \mathbf p \|^{2} w_{n})_{n\in\mathbb Z}.
\end{equation}
Consider the eigenvalue equation
\begin{equation}
L_{B,\mathbf q}^{\alpha}(w_{n})=\lambda (w_{n}).
\end{equation}
Using \eqref{LBq-sg2} this can be rewritten as
\begin{equation}\label{evns-sg}
\rho_{n-1}'w_{n-1}-\rho_{n+1}'w_{n+1}=\lambda w_{n}+ \nu (1+\alpha^{2}\|\mathbf q + n \mathbf p \|^{2})^{-1}\|\mathbf q + n \mathbf p\|^{2}w_{n}.
\end{equation}
Letting $z_{n}=\rho_{n}'w_{n}$ and introducing the notation
\begin{equation}
e_{n}(\lambda,\nu,\alpha)=\frac{\lambda}{\rho_{n}'}+\frac{\nu \|\mathbf q + n \mathbf p\|^{2}}{(1+\alpha^{2}\|\mathbf q + n \mathbf p \|^{2}) \rho_{n}'}
\end{equation}
(notice that $e_{n} \to \lambda +\nu/\alpha^{2}$ as $n \to \pm \infty$) we rewrite \eqref{evns-sg} as 
\begin{equation}
z_{n-1}-z_{n+1}=e_{n} z_{n}.
\end{equation}
Now, as before, let $u_{n}=z_{n-1}/z_{n}$ to obtain 
\begin{equation}\label{unformula1sg}
u_{n}=e_{n}+\frac{1}{u_{n+1}}, \quad n \in \mathbb Z.
\end{equation}
Iterating this equation above forwards for $n \geq 0$, one sees that for each $n \geq 0$, $u_{n}$ must satisfy the following continued fraction
\begin{equation}\label{un1sg}
u_{n}=u_{n}^{(1)}(\lambda,\nu,\alpha)= e_{n}+[e_{n+1};e_{n+2};\ldots], \, \quad n=0,1,2,\ldots.
\end{equation}
Let $e_{\infty}:=\lim_{n \to \pm \infty} e_{n}=\lambda+\nu/\alpha^{2}$ and denote $u_{\pm \infty}=e_{\infty}/2 \pm \sqrt{(e_{\infty}/2)^{2}+1}$ and note that $u_{n}^{(1)} \to u_{+\infty}$ as $n \to \infty$ and moreover $u_{+\infty}>1$. (The proof of this fact is similar to the Euler-$\alpha$ case considered in \cite{DLMVW20}).
Similarly, one has 
\begin{equation}\label{unformula2-sg}
u_{n+1}=\frac{-1}{e_{n}-u_{n}}, \quad n \in \mathbb Z. 
\end{equation}
Iterating this for $n \leq 0$, one has, 
\begin{equation}\label{un2-sg}
u_{n}=u_{n}^{(2)}(\lambda,\nu,\alpha)= -[e_{n-1};e_{n-2};\ldots], \quad n=0,-1,-2,\ldots.
\end{equation}
Notice that $u_{n} \to u_{-\infty}$ as $n \to -\infty$ and $-1< u_{-\infty}<0$. Since the expressions for $u_{0}$, given respectively, by equations \eqref{un1sg} and \eqref{un2-sg} must match, we must have $u_{0}^{(1)}(\lambda,\nu)=u_{0}^{(2)}(\lambda,\nu)$, i.e.,
\begin{equation}\label{un12-sg}
u_{0}^{(1)}(\lambda,\nu,\alpha)= e_{0}+[e_{1};e_{2};\ldots]=u_{0}^{(2)}(\lambda,\nu,\alpha)=-[e_{-1};e_{-2};\ldots].
\end{equation}
Denote 
\begin{equation}\label{f-sg}
f^{s}(\lambda, \nu,\alpha)= [e_{1};e_{2};\ldots]
\end{equation}
and
\begin{equation}\label{g-sg}
g^{s}(\lambda, \nu,\alpha)=[e_{-1};e_{-2};\ldots].
\end{equation}
Notice that
\begin{equation}\label{fgu012-sg}
f^{s}(\lambda,\nu,\alpha)=u_{0}^{(1)}(\lambda,\nu,\alpha)-e_{0}(\lambda,\nu), \quad g^{s}(\lambda,\nu,\alpha)=-u_{0}^{(2)}(\lambda,\nu,\alpha).
\end{equation}
Equation \eqref{un12-sg} is equivalent to 
\begin{equation}\label{fg-sg}
e_{0}(\lambda, \nu,\alpha)+f^{s}(\lambda,\nu,\alpha)+g^{s}(\lambda,\nu,\alpha)=0.
\end{equation}
Thus, the eigenvalue problem \eqref{evns-sg} has an unstable eigenvalue $\lambda > 0$ if and only if equation \eqref{fg-sg} has a root $\lambda > 0$. If $\mathbf q$ is of type $I_{0}$, $\rho_{0}'<0$ and $\rho_{n}'>0$ for every $n \neq 0$. $-e_{0}$ as a function of $\lambda$ is a straight line of slope $-1/\rho_{0}'$ and $y$-intercept $-\frac{ \nu \|\mathbf q \|^{2}}{(1+\alpha^{2}\|\mathbf q  \|^{2}) \rho_{0}'}$. Moreover, we can show $f^{s}$ and $g^{s}$ considered as functions of $\lambda$ are continuous on $[0,+\infty)$ with limits $\lim_{\lambda \to +\infty}f^{\alpha}(\lambda, \nu,\alpha)=g^{\alpha}(\lambda, \nu,\alpha)=0$.
Thus the two curves will meet provided that $\lim_{\lambda \to 0^{+}}f^{s}(\lambda, \nu,\alpha)+g^{s}(\lambda, \nu,\alpha)$ is greater than the $y$-intercept of the line $-e_{0}$: $-\frac{ \nu \|\mathbf q\|^{2} }{(1+\alpha^{2}\|\mathbf q  \|^{2}) \rho_{0}'}$.

One can show, similar to Lemma \ref{ff0} in Section 2, that $\lim_{\lambda \to 0^{+}}f^{s}(\lambda, \nu,\alpha)=f^{s}(0,\nu,\alpha)$ and $\lim_{\lambda \to 0^{+}}g^{s}(\lambda, \nu,\alpha)=g^{s}(0,\nu,\alpha)$ which then means that we must check that the inequality
\begin{equation}\label{falphaineq}
f^{s}(0,\nu,\alpha)+g^{s}(0,\nu,\alpha) >-\frac{ \nu \|\mathbf q \|^{2}}{(1+\alpha^{2}\|\mathbf q \|^{2}) \rho_{0}'}
\end{equation}
holds. Considered as a function of $\nu$, $-\frac{\nu \|\mathbf q \|^{2}}{(1+\alpha^{2}\|\mathbf q  \|^{2}) \rho_{0}'}$ is a straight line with positive slope that passes through the origin $\nu=0$. The functions $f^{s}(0,\nu,\alpha)$ and $g^{s}(0,\nu,\alpha)$ are continuous in $\nu$ and satisfy limits
$\lim_{\nu \to +\infty}f^{s}(0,\nu,\alpha)=g^{s}(0,\nu,\alpha)=0$. %
If we therefore show that $\lim_{\nu \to 0^{+}} f^{s}(0,\nu,\alpha) > 0$ and $\lim_{\nu \to 0^{+}} g^{s}(0,\nu,\alpha) > 0$ then there will exist a $\nu_{0}>0$ such that for all $\nu \in (0,\nu_{0})$ the inequality in \eqref{falphaineq} is satisfied. 
\begin{lemma}\label{nuzerolemma}
The following limits hold
\begin{equation}
\lim_{\nu \to 0^{+}} f^{s}(0,\nu,\alpha)=1,
\end{equation}
\begin{equation}
\lim_{\nu \to 0^{+}} g^{s}(0,\nu,\alpha)=1.
\end{equation}
\end{lemma}
The limits above follow from the following general fact from the theory of continued fractions. Let $(c_{n})$ be a sequence of positive numbers that converge to $1$ and consider the continued fraction $f(x)=[xc_{1};xc_{2};\ldots]$. Then $\lim_{x \to 0^{+}}f(x)=1$. For a proof, see Lemma 3.1, Item (4) in \cite{DLMVW20}. The above discussion thus guarantees the existence of a positive root $\lambda >0$ to the equation \eqref{fg-sg}. The remaining parts are similar to the Navier-Stokes case. We can thus establish the main instability Theorem \ref{mainthm-sg-nsa-nsv}.
Extensions to cover the cases when $\mathbf q$ is of type $I_{+}$ and $I_{-}$ can be carried out similar to those described in Subsection 2.1.

\subsection{Navier-Stokes-$\alpha$}
We now consider the Navier-Stokes-$\alpha$ equations.
Taking the curl of the Navier-Stokes-$\alpha$ equation \eqref{nsalphavel} and setting $\omega=\curl(1-\alpha^{2}\Delta)\mathbf u_{f}$ will yield the equation
\begin{equation}\label{nsalphavort}
\partial_{t}\omega+\mathbf u_{f}\cdot \nabla \omega=\nu \Delta \omega+\nu f.
\end{equation}
Denote by $\varphi$, the stream function corresponding to $\mathbf u_{f}$, i.e., $\mathbf u_{f}=(\varphi_{y},-\varphi_{x})$. As before, we have that $\curl \mathbf u_{f}=\curl (\varphi_{y},-\varphi_{x})=-\Delta \varphi$. Thus $\omega=\curl(1-\alpha^{2}\Delta)\mathbf u_{f}=-\Delta(I-\alpha^{2}\Delta)\varphi$.
Fix $\mathbf p \in \mathbb Z^{2}$, where $\mathbf p=(p_{1},p_{2})$ and denote $\mathbf p^{\perp}=(-p_{2},p_{1})$. Also, fix $\Gamma \in \mathbb R$ and  consider a steady state to \eqref{nsalphavort} of the form
\begin{align}\label{steadystatensalpha}
\omega^{0}=\Gamma \cos (\mathbf p \cdot \mathbf x), \quad \mathbf u_{f}^{0}=\frac{\Gamma \sin (\mathbf p \cdot \mathbf x) \mathbf p^{\perp}}{\|\mathbf p\|^{2}(1+\alpha^{2}\|\mathbf p\|^{2})}, \nonumber \\
f=-\Delta \omega^{0}=\|\mathbf p\|^{2}\Gamma \cos (\mathbf p \cdot \mathbf x).
\end{align}
The stream function $\varphi^{0}$ that corresponds to $\omega^{0}$ is given by the formula \[\varphi^{0}=\frac{\Gamma \cos (\mathbf p \cdot \mathbf x)}{\|\mathbf p\|^{2}(1+\alpha^{2}\|\mathbf p\|^{2})}.\]
Using the Fourier series decomposition  $\omega(\mathbf x)=\sum_{\mathbf k\in\mathbb Z^2\setminus\{0\}}\omega_\mathbf k e^{i \mathbf k\cdot\mathbf x}$  
 we rewrite \eqref{nsalphavort} as
\begin{equation}\label{nsafourier}
\frac{d\omega_{\mathbf k}}{dt}=\sum_{\mathbf q\in\mathbb Z^2\setminus\{0\}}\beta_{1}(\mathbf k-\mathbf q,\mathbf q)\omega_{\mathbf k-\mathbf q}\omega_\mathbf q - \nu \|\mathbf k\|^{2}\omega_{\mathbf k} + \nu f_{\mathbf k} ,\,\,\mathbf k\in\mathbb Z^2\setminus\{0\},
\end{equation}
where the coefficients $\beta_{1}(\mathbf p,\mathbf q)$ for $\mathbf p,\mathbf q\in\mathbb Z^2$ are defined as 
\begin{equation}\label{dfnnsa}
\beta_{1}(\mathbf p,\mathbf q)=\frac{1}{2}\bigg(\|\mathbf q\|^{-2}(1+\alpha^{2}\|\mathbf q\|^{2})^{-1}-\|\mathbf p\|^{-2}(1+\alpha^{2}\|\mathbf p\|^{2})^{-1}\bigg)(\mathbf p\wedge\mathbf q)\,
\end{equation}
for  
$\mathbf p\neq0,\mathbf q\neq0$, and $\beta_{1}(\mathbf p,\mathbf q)=0$ otherwise.

Consider the linearization of \eqref{nsalphavort} about the steady state \eqref{steadystatensalpha} given by
\begin{equation}\label{nsalin}
\partial_{t}\omega + \mathbf u_{f}^{0} \cdot \nabla \omega + \mathbf u_{f} \cdot \nabla \omega^{0}= \nu \Delta  \omega.
\end{equation} 
Corresponding to \eqref{nsalin}, consider the linear operator $L_{B}^{\alpha}$ given by
\begin{equation}
L_{B}^{\alpha} \omega = - \mathbf u_{f}^{0} \cdot \nabla \omega - \mathbf u_{f} \cdot \nabla \omega^{0}+ \nu \Delta  \omega.
\end{equation}
Using the Fourier decomposition in \eqref{nsafourier}, consider the following linearized vorticity operator in the space $\ell^2(\mathbb Z^2)$
\begin{align}
L_B^{\alpha}&: (\omega_\mathbf k)_{\mathbf k\in\mathbb Z^2}\mapsto\nonumber\\&
\big(\beta_{1}(\mathbf p,\mathbf k-\mathbf p)\Gamma\omega_{\mathbf k-\mathbf p}-
\beta_{1}(\mathbf p,\mathbf k+\mathbf p)\Gamma \omega_{\mathbf k+\mathbf p}- \nu \|\mathbf k\|^{2}\omega_{\mathbf k}\big)_{\mathbf k\in\mathbb Z^2}.\label{dfnLBa-nsa}
\end{align}
Fixing a $\mathbf q$ of type $I_{0}$ and decomposing spaces and operators as before, consider the operator $L_{B,\mathbf q}^{\alpha}$ acting on the space $\ell_{2}(\mathbb Z)$ given by
\begin{align}
L_{B,\mathbf q}^{\alpha}: (w_n)_{n\in\mathbb Z}\mapsto
\bigg( \beta_{1}&(\mathbf p,\mathbf q+(n-1)\mathbf p)\Gamma w_{n-1}-
\beta_{1}(\mathbf p,\mathbf q+(n+1)\mathbf p)\Gamma w_{n+1} \nonumber \\&- \nu \|\mathbf q + n \mathbf p \|^{2} w_{n} \bigg)_{n\in\mathbb Z}.\label{nsaLBq}
\end{align}
(Notice that the $L_B^{\alpha}$ and $L_{B,\mathbf q}^{\alpha}$ considered here differ in the viscous term as compared to the second grade fluid model in the previous subsection.) 
Let $\rho_{n}'$ be defined by the equation
\begin{equation}\label{rhonnsa}
\rho_{n}'=\Gamma \beta_{1}(\mathbf p,\mathbf q+n\mathbf p)=\frac{(\mathbf q \wedge \mathbf p) \Gamma}{2 (1+\alpha^{2}\|\mathbf p\|^{2})\|\mathbf p\|^{2}}\bigg( 1-\frac{(1+\alpha^{2}\|\mathbf p\|^{2})\|\mathbf p\|^{2}}{(1+\alpha^{2}\|\mathbf q+n \mathbf p\|^{2})\|\mathbf q+n \mathbf p\|^{2}}\bigg).
\end{equation}
Normalize $\Gamma$ so that 
\begin{equation}\label{gnormnsa}
\frac{(\mathbf q \wedge \mathbf p) \Gamma}{2 (1+\alpha^{2}\|\mathbf p\|^{2})\|\mathbf p\|^{2}}=1.
\end{equation}
After this normalization, $\rho_{n}' \to 1$ as $n \to \pm \infty$. Using \eqref{rhonnsa} the operator $L_{B,\mathbf q}^{\alpha}$ can be rewritten as 
\begin{equation}\label{LBq-nsa2}
L_{B,\mathbf q}^{\alpha}: (w_n)_{n\in\mathbb Z}\mapsto (\rho_{n-1}'w_{n-1}-\rho_{n+1}'w_{n+1}- \nu \|\mathbf q + n \mathbf p \|^{2} w_{n})_{n\in\mathbb Z}.
\end{equation}
Consider the eigenvalue equation
\begin{equation}
L_{B,\mathbf q}^{\alpha}(w_{n})=\lambda (w_{n}).
\end{equation}
Using \eqref{LBq-nsa2} this can be rewritten as
\begin{equation}\label{evns-nsa}
\rho_{n-1}'w_{n-1}-\rho_{n+1}'w_{n+1}=\lambda w_{n}+ \nu \|\mathbf q + n \mathbf p\|^{2}w_{n}.
\end{equation}
Letting $z_{n}=\rho_{n}'w_{n}$ and introducing the notation
\begin{equation}
l_{n}=l_{n}(\lambda,\nu,\alpha)=\frac{\lambda}{\rho_{n}'}+\frac{\nu \|\mathbf q + n \mathbf p\|^{2}}{\rho_{n}'}
\end{equation}
we rewrite \eqref{evns-nsa} as 
\begin{equation}
z_{n-1}-z_{n+1}=l_{n} z_{n}.
\end{equation}
Now, as before, let $u_{n}=z_{n-1}/z_{n}$ to obtain 
\begin{equation}\label{unformula1nsa}
u_{n}=l_{n}+\frac{1}{u_{n+1}}, \quad n \in \mathbb Z.
\end{equation}
Iterating this equation above forwards for $n \geq 0$, one sees that for each $n \geq 0$, $u_{n}$ must satisfy the following continued fraction
\begin{equation}\label{un1nsa}
u_{n}=u_{n}^{(1)}(\lambda,\nu,\alpha)= l_{n}+[l_{n+1};l_{n+2};\ldots],\quad n=0,1,2,\ldots.
\end{equation}
Similar to the Navier-Stokes case $\lim_{n \to \pm \infty} l_{n}=+\infty $ and thus $\lim_{n \to +\infty} u_{n}^{(1)}(\lambda,\nu,\alpha)=+\infty$. 
Similarly, one has 
\begin{equation}\label{unformula2-nsa}
u_{n+1}=\frac{-1}{l_{n}-u_{n}}, \quad n \in \mathbb Z. 
\end{equation}
Iterating this for $n \leq 0$, one has, 
\begin{equation}\label{un2-nsa}
u_{n}=u_{n}^{(2)}(\lambda,\nu,\alpha)= -[l_{n-1};l_{n-2};\ldots], \quad n=0,-1,-2,\ldots.
\end{equation}
Notice that $\lim_{n \to -\infty}u_{n}^{(2)}(\lambda,\nu,\alpha) = 0$. Since the expressions for $u_{0}$, given respectively, by equations \eqref{un1nsa} and \eqref{un2-nsa} must match, we must have $u_{0}^{(1)}(\lambda,\nu)=u_{0}^{(2)}(\lambda,\nu)$, i.e.,
\begin{equation}\label{un12-nsa}
u_{0}^{(1)}(\lambda,\nu,\alpha)= l_{0}+[l_{1};l_{2};\ldots]=u_{0}^{(2)}(\lambda,\nu,\alpha)=-[l_{-1};l_{-2};\ldots].
\end{equation}
Denote 
\begin{equation}\label{f-nsa}
f^{a}(\lambda, \nu,\alpha)= [l_{1}(\lambda, \nu,\alpha);l_{2}(\lambda, \nu,\alpha);\ldots]
\end{equation}
and
\begin{equation}\label{g-nsa}
g^{a}(\lambda, \nu,\alpha)=[l_{-1}(\lambda, \nu,\alpha);l_{-2}(\lambda, \nu,\alpha);\ldots].
\end{equation}
Notice that
\begin{equation}\label{fgu012-nsa}
f^{a}(\lambda,\nu,\alpha)=u_{0}^{(1)}(\lambda,\nu,\alpha)-l_{0}(\lambda,\nu), \quad g^{a}(\lambda,\nu,\alpha)=-u_{0}^{(2)}(\lambda,\nu,\alpha).
\end{equation}
Equation \eqref{un12-nsa} is equivalent to 
\begin{equation}\label{fg-nsa}
l_{0}(\lambda, \nu,\alpha)+f^{a}(\lambda,\nu,\alpha)+g^{a}(\lambda,\nu,\alpha)=0.
\end{equation}
Thus, the eigenvalue problem \eqref{evns-nsa} has an unstable eigenvalue $\lambda > 0$ if and only if equation \eqref{fg-nsa} has a root $\lambda > 0$. If $\mathbf q$ is of type $I_{0}$, $\rho_{0}'<0$ and $\rho_{n}'>0$ for every $n \neq 0$. $-l_{0}$ as a function of $\lambda$ is a straight line of slope $-1/\rho_{0}'$ and $y$-intercept $-\frac{\nu \|\mathbf q \|^{2}}{ \rho_{0}'}$. Moreover, we can show $f^{a}$ and $g^{a}$ considered as functions of $\lambda$ are continuous on $[0,+\infty)$ with limits $\lim_{\lambda \to +\infty}f^{a}(\lambda, \nu,\alpha)=g^{a}(\lambda, \nu,\alpha)=0$. 
Thus the two curves will meet provided that $\lim_{\lambda \to 0^{+}}f^{a}(\lambda, \nu,\alpha)+g^{a}(\lambda, \nu,\alpha)$ is greater than the $y$-intercept of the line $-l_{0}$: $-\frac{\nu \|\mathbf q \|^{2}}{\rho_{0}'}$.

One can show, similar to Lemma \ref{ff0} in Section 2, that $\lim_{\lambda \to 0^{+}}f^{a}(\lambda, \nu,\alpha)=f^{a}(0,\nu,\alpha)$ and $\lim_{\lambda \to 0^{+}}g^{a}(\lambda, \nu,\alpha)=g^{a}(0,\nu,\alpha)$ which then means that we must check that the inequality
\begin{equation}\label{falphaineq-nsa}
f^{a}(0,\nu,\alpha)+g^{a}(0,\nu,\alpha) >-\frac{\nu \|\mathbf q \|^{2}}{\rho_{0}'}
\end{equation}
holds. 

The strategy for proving inequality \eqref{falphaineq-nsa} is similar to the Navier-Stokes case. Similarly to Lemma \ref{even-trunc-properties} and Lemma \ref{nu0} we can show that there exists a positive integer $k$ such that the slopes of the even truncations $f^{a}(0,\nu,\alpha)^{2k}$ and $g^{a}(0,\nu,\alpha)^{2k}$ at $\nu=0$ is greater than the slope $-\frac{ \|\mathbf q \|^{2}}{\rho_{0}'}$ of the straight line $-\frac{\nu \|\mathbf q \|^{2}}{\rho_{0}'}$. This will then establish that there exists a $\nu_{0}>0$ such that for all $\nu \in (0,\nu_{0})$ the inequality \eqref{falphaineq-nsa} is satisfied. The other facts can be proved similarly to the Navier-Stokes case. We can thus establish Theorem \ref{mainthm-sg-nsa-nsv}.

\subsection{Navier-Stokes-Voigt model} 
We finally consider the Navier-Stokes-Voigt model. Taking $\curl$ of equation \eqref{nsvoigtvel} and setting 
\begin{equation}\label{vortformulansvoigt}
\omega=\curl(I-\alpha^{2}\Delta)\mathbf u_{f},
\end{equation}
we obtain the equation
\begin{equation}\label{nsvoigtvort}
\partial_{t}\omega+\mathbf u_{f}\cdot \nabla (I-\alpha^{2}\Delta)^{-1}\omega=\nu \Delta (I-\alpha^{2}\Delta)^{-1}\omega+\nu f.
\end{equation}
Denote by $\varphi$, the stream function corresponding to $\mathbf u$, i.e., $\mathbf u_{f}=(\varphi_{y},-\varphi_{x})$. Note first that $\curl \mathbf u_{f}=\curl (\varphi_{y},-\varphi_{x})=-\Delta \varphi$. We thus have that $\omega=\curl (I-\alpha^{2}\Delta)\mathbf u=-\Delta (I-\alpha^{2}\Delta)\varphi$.
Fix $\mathbf p \in \mathbb Z^{2}$, where $\mathbf p=(p_{1},p_{2})$ and denote $\mathbf p^{\perp}=(-p_{2},p_{1})$. Also, fix $\Gamma \in \mathbb R$ and  consider a steady state to \eqref{nsvoigtvort} of the form
\begin{align}\label{steadystatensvoigt}
\omega^{0}=\Gamma \cos (\mathbf p \cdot \mathbf x), \quad \mathbf u_{f}^{0}=\frac{\Gamma \sin (\mathbf p \cdot \mathbf x) \mathbf p^{\perp}}{(1+\alpha^{2}\|\mathbf p\|^{2})\|\mathbf p\|^{2}}, \nonumber \\
f=-(I-\alpha^{2}\Delta)^{-1}\Delta \omega^{0}=\frac{\|\mathbf p\|^{2}\Gamma \cos (\mathbf p \cdot \mathbf x)}{1+\alpha^{2}\|\mathbf p\|^{2}}.
\end{align}
The stream function $\varphi^{0}$ that corresponds to $\omega^{0}$ is given by the formula \[\varphi^{0}=\frac{\Gamma \cos (\mathbf p \cdot \mathbf x)}{\|\mathbf p\|^{2}(1+\alpha^{2}\|\mathbf p\|^{2})}.\]
Using the Fourier series decomposition  $\omega(\mathbf x)=\sum_{\mathbf k\in\mathbb Z^2\setminus\{0\}}\omega_\mathbf k e^{i \mathbf k\cdot\mathbf x}$  
we rewrite \eqref{nsvoigtvort} as
\begin{equation}\label{nsvfourier}
\frac{d \omega_{\mathbf k}}{dt}=\sum_{\mathbf q\in\mathbb Z^2\setminus\{0\}}\beta_{2}(\mathbf k-\mathbf q,\mathbf q)\omega_{\mathbf k-\mathbf q}\omega_\mathbf q - \nu \frac{\|\mathbf k\|^{2}}{1+\alpha^{2}\|\mathbf k\|^{2}} \omega_{\mathbf k} + \nu f_{\mathbf k} ,\,\,\mathbf k\in\mathbb Z^2\setminus\{0\},
\end{equation}
where the coefficients $\beta(\mathbf p,\mathbf q)$ for $\mathbf p,\mathbf q\in\mathbb Z^2$ are defined as 
\begin{equation}\label{dfnnsv}
\beta_{2}(\mathbf p,\mathbf q)=\frac{1}{2}\bigg(\|\mathbf q\|^{-2}-\|\mathbf p\|^{-2}\bigg)\frac{(\mathbf p\wedge\mathbf q)}{(1+\alpha^{2}\|\mathbf p\|^{2})(1+\alpha^{2}\|\mathbf q\|^{2})}\,
\end{equation}
for %
$\mathbf p\neq0,\mathbf q\neq0$, and $\beta_{1}(\mathbf p,\mathbf q)=0$ otherwise.

Consider the linearization of \eqref{nsvoigtvort} about the steady state \eqref{steadystatensvoigt} given by
\begin{equation}\label{nsvlin}
\partial_{t} \omega + \mathbf u_{f}^{0} \cdot \nabla \omega + \mathbf u_{f} \cdot \nabla \omega^{0}= \nu \Delta (I-\alpha^{2}\Delta)^{-1}\omega.
\end{equation} 
Corresponding to \eqref{nsvlin}, consider the linear operator $L_{B}^{\alpha}$ given by
\begin{equation}
L_{B}^{\alpha} \omega = - \mathbf u_{f}^{0} \cdot \nabla \omega - \mathbf u_{f} \cdot \nabla \omega^{0}+ \nu \Delta (I-\alpha^{2}\Delta)^{-1} \omega.
\end{equation}
Using the Fourier decomposition in \eqref{nsvfourier}, consider the following linearized vorticity operator in the space $\ell^2(\mathbb Z^2)$
\begin{align}
L_B^{\alpha}&: (\omega_\mathbf k)_{\mathbf k\in\mathbb Z^2}\mapsto\nonumber\\&
(\beta_{2}(\mathbf p,\mathbf k-\mathbf p)\Gamma\omega_{\mathbf k-\mathbf p}-
\beta_{2}(\mathbf p,\mathbf k+\mathbf p)\Gamma \omega_{\mathbf k+\mathbf p}- \nu \frac{\|\mathbf k\|^{2}}{1+\alpha^{2}\|\mathbf k\|^{2}}\omega_{\mathbf k})_{\mathbf k\in\mathbb Z^2}.\label{dfnLBa-nsv}
\end{align}
Fixing a $\mathbf q$ of type $I_{0}$ and decomposing spaces and operators as before, consider the operator $L_{B,\mathbf q}^{\alpha}$ acting on the space $\ell_{2}(\mathbb Z)$ given by
\begin{align}
L_{B,\mathbf q}^{\alpha}: &(w_n)_{n\in\mathbb Z}\mapsto
 \bigg(\beta_{2}(\mathbf p,\mathbf q+(n-1)\mathbf p)\Gamma w_{n-1} \nonumber \\&-
\beta_{2}(\mathbf p,\mathbf q+(n+1)\mathbf p)\Gamma w_{n+1}  - \nu \|\mathbf q + n \mathbf p \|^{2}(1+\alpha^{2}\|\mathbf q + n \mathbf p \|^{2})^{-1} w_{n}\bigg)_{n\in\mathbb Z}.\label{nsvLBq}
\end{align}
Let $\rho_{n}$ be defined by the equation
\begin{equation}\label{rhonnsv}
\rho_{n}=\Gamma \beta_{2}(\mathbf p,\mathbf q+n\mathbf p)=\frac{(\mathbf q \wedge \mathbf p) \Gamma}{2 \|\mathbf p\|^{2}(1+\alpha^{2}\|\mathbf p\|^{2})(1+\alpha^{2}\|\mathbf q+n\mathbf p\|^{2})}\bigg( 1-\frac{\|\mathbf p\|^{2}}{\|\mathbf q+n \mathbf p\|^{2}}\bigg).
\end{equation}
Normalize $\Gamma$ so that 
\begin{equation}\label{gnormnsv}
\frac{(\mathbf q \wedge \mathbf p) \Gamma}{2 \|\mathbf p\|^{2}(1+\alpha^{2}\|\mathbf p\|^{2})(1+\alpha^{2}\|\mathbf q+n\mathbf p\|^{2})}=1.
\end{equation}
After this normalization, $\rho_{n} \to 1$ as $n \to \pm \infty$. Using \eqref{rhonnsv} the operator $L_{B,\mathbf q}^{\alpha}$ can be rewritten as 
\begin{equation}\label{LBq-nsv2}
L_{B,\mathbf q}^{\alpha}: (w_n)_{n\in\mathbb Z}\mapsto (\rho_{n-1}w_{n-1}-\rho_{n+1}w_{n+1}- \nu \|\mathbf q + n \mathbf p \|^{2}(1+\alpha^{2}\|\mathbf q + n \mathbf p \|^{2})^{-1} w_{n})_{n\in\mathbb Z}.
\end{equation}
Consider the eigenvalue equation
\begin{equation}
L_{B,\mathbf q}^{\alpha}(w_{n})=\lambda (w_{n}).
\end{equation}
Using \eqref{LBq-nsv2} this can be rewritten as
\begin{equation}\label{evns-nsv}
\rho_{n-1}w_{n-1}-\rho_{n+1}w_{n+1}=\lambda w_{n}+ \nu \|\mathbf q + n \mathbf p\|^{2}(1+\alpha^{2}\|\mathbf q + n \mathbf p \|^{2})^{-1}w_{n}.
\end{equation}
Letting $z_{n}=\rho_{n}w_{n}$ and introducing the notation
\begin{equation}
i_{n}(\lambda,\nu,\alpha)=\frac{\lambda}{\rho_{n}}+\frac{\nu \|\mathbf q + n \mathbf p\|^{2}}{(1+\alpha^{2}\|\mathbf q + n \mathbf p \|^{2})\rho_{n}}
\end{equation}
we rewrite \eqref{evns-nsv} as 
\begin{equation}
z_{n-1}-z_{n+1}=i_{n} z_{n}.
\end{equation}
Now, let $u_{n}=z_{n-1}/z_{n}$ to obtain 
\begin{equation}\label{unformula1nsv}
u_{n}=i_{n}+\frac{1}{u_{n+1}}, \quad n \in \mathbb Z.
\end{equation}
Iterating this equation above forwards for $n \geq 0$, one sees that for each $n \geq 0$, $u_{n}$ must satisfy the following continued fraction
\begin{equation}\label{un1nsv}
u_{n}=u_{n}^{(1)}(\lambda,\nu,\alpha)= i_{n}+[i_{n+1};i_{n+2};\ldots]\quad n=0,1,2,\ldots.
\end{equation}
Similarly, one has 
\begin{equation}\label{unformula2-nsv}
u_{n+1}=\frac{-1}{i_{n}-u_{n}}, \quad n \in \mathbb Z. 
\end{equation}
Iterating this for $n \leq 0$, one has,
\begin{equation}\label{un2-nsv}
u_{n}=u_{n}^{(2)}(\lambda,\nu,\alpha)= -[i_{n-1};i_{n-2};\ldots]\quad n=0,-1,-2,\ldots.
\end{equation} 
Since the expressions for $u_{0}$, given respectively, by equations \eqref{un1nsv} and \eqref{un2-nsv} must match, we must have $u_{0}^{(1)}(\lambda,\nu)=u_{0}^{(2)}(\lambda,\nu)$, i.e.,
\begin{equation}\label{un12-nsv}
u_{0}^{(1)}(\lambda,\nu,\alpha)= i_{0}+[i_{1};i_{2};\ldots]=u_{0}^{(2)}(\lambda,\nu,\alpha)=-[i_{-1};i_{-2};\ldots].
\end{equation}
Denote 
\begin{equation}\label{f-nsv}
f^{v}(\lambda, \nu,\alpha)= [i_{1}(\lambda, \nu,\alpha);i_{2}(\lambda, \nu,\alpha);\ldots]
\end{equation}
and
\begin{equation}\label{g-nsv}
g^{v}(\lambda, \nu,\alpha)=[i_{-1}(\lambda, \nu,\alpha);i_{-2}(\lambda, \nu,\alpha);\ldots].
\end{equation}
Notice that
\begin{equation}\label{fgu012-nsv}
f^{v}(\lambda,\nu,\alpha)=u_{0}^{(1)}(\lambda,\nu,\alpha)-i_{0}(\lambda,\nu), \quad g^{v}(\lambda,\nu,\alpha)=-u_{0}^{(2)}(\lambda,\nu,\alpha).
\end{equation}
Equation \eqref{un12-nsv} is equivalent to 
\begin{equation}\label{fg-nsv}
i_{0}(\lambda, \nu,\alpha)+f^{v}(\lambda,\nu,\alpha)+g^{v}(\lambda,\nu,\alpha)=0.
\end{equation}
Thus, the eigenvalue problem \eqref{evns-nsv} has an unstable eigenvalue $\lambda > 0$ if and only if equation \eqref{fg-nsv} has a root $\lambda > 0$. If $\mathbf q$ is of type $I_{0}$, $\rho_{0}<0$ and $\rho_{n}>0$ for every $n \neq 0$. $-i_{0}$ as a function of $\lambda$ is a straight line of slope

 $-1/\rho_{0}$ and $y$-intercept $-\frac{ \nu \|\mathbf q \|^{2}}{(1+\alpha^{2}\|\mathbf q  \|^{2}) \rho_{0}}$. Moreover, we can show $f^{v}$ and $g^{v}$ considered as functions of $\lambda$ are continuous on $[0,+\infty)$ with limits $\lim_{\lambda \to +\infty}f^{v}(\lambda, \nu,\alpha)=g^{v}(\lambda, \nu,\alpha)=0$.
Thus the two curves will meet provided that $\lim_{\lambda \to 0^{+}}f^{v}(\lambda, \nu,\alpha)+g^{v}(\lambda, \nu,\alpha)$ is greater than the $y$-intercept of the line $-i_{0}$: $-\frac{ \nu \|\mathbf q\|^{2} }{(1+\alpha^{2}\|\mathbf q  \|^{2}) \rho_{0}}$.

One can show, similar to Lemma \ref{ff0} in Section 2, that $\lim_{\lambda \to 0^{+}}f^{v}(\lambda, \nu,\alpha)=f^{v}(0,\nu,\alpha)$ and $\lim_{\lambda \to 0^{+}}g^{v}(\lambda, \nu,\alpha)=g^{v}(0,\nu,\alpha)$ which then means that we must check that the inequality
\begin{equation}\label{falphaineq-voigt}
f^{v}(0,\nu,\alpha)+g^{v}(0,\nu,\alpha) >-\frac{ \nu \|\mathbf q \|^{2}}{(1+\alpha^{2}\|\mathbf q \|^{2}) \rho_{0}}
\end{equation}
holds. Considered as a function of $\nu$, $-\frac{\nu \|\mathbf q \|^{2}}{(1+\alpha^{2}\|\mathbf q  \|^{2}) \rho_{0}}$ is a straight line with positive slope that passes through the origin $\nu=0$. The functions $f^{v}(0,\nu,\alpha)$ and $g^{v}(0,\nu,\alpha)$ are continuous in $\nu$ and satisfy limits
$\lim_{\nu \to +\infty}f^{v}(0,\nu,\alpha)=g^{v}(0,\nu,\alpha)=0$. %
If we therefore show that $\lim_{\nu \to 0^{+}} f^{v}(0,\nu,\alpha) > 0$ and $\lim_{\nu \to 0^{+}} g^{v}(0,\nu,\alpha) > 0$ then there will exist a $\nu_{0}>0$ such that for all $\nu \in (0,\nu_{0})$ the inequality in \eqref{falphaineq} is satisfied. Similar to the second grade fluid model, we can prove the following Lemma.
\begin{lemma}\label{nuzerolemma}
The following limits hold
\begin{equation}
\lim_{\nu \to 0^{+}} f^{v}(0,\nu,\alpha)=1,
\end{equation}
\begin{equation}
\lim_{\nu \to 0^{+}} g^{v}(0,\nu,\alpha)=1.
\end{equation}
\end{lemma}
The above discussion thus guarantees the existence of a positive root $\lambda >0$ to the equation \eqref{fg-sg}. The remaining parts are similar to the Navier-Stokes case. We can thus establish Theorem \ref{mainthm-sg-nsa-nsv}.

\begin{theorem}\label{mainthm-sg-nsa-nsv}
Suppose \eqref{steadystatesecondgrade},\eqref{steadystatensalpha}, \eqref{steadystatensvoigt} are, respectively, steady state solutions to the 2D second grade equations \eqref{secondgradevort}, the Navier-Stokes-$\alpha$ equations \eqref{nsalphavort} and the Navier-Stokes-Voigt equations \eqref{nsvoigtvort} such that there exists at least one point 
$\bq \in \cQ(\bp)$ of type $I_{0}$, where $\bq$ is not parallel to $\bp$. Also, we assume that $\Gamma \in \mathbb R$ and satisfies the respective normalization conditions given by equations \eqref{gnormsg}, \eqref{gnormnsa} and \eqref{gnormnsv}.
Then there exist $\nu_{0}^{sg}, \nu_{0}^{a}, \nu_{0}^{v}$ such that for all $\nu$ in the respective intervals $(0,\nu_{0}^{sg})$, $(0,\nu_{0}^{a})$, $(0,\nu_{0}^{v})$ the steady state $\omega^{0}$, $\mathbf u^{0}$ and $f$ defined in equations  \eqref{steadystatesecondgrade}, \eqref{steadystatensalpha} and \eqref{steadystatensvoigt} are linearly unstable. In particular, the respective linear operators $L_{B,\mathbf q}^{\alpha}$ in the space $\ell^{2}_{s}(\mathbb Z)$ have a positive eigenvalue and therefore the operators $L_{B}^{\alpha}$ in $\ell_{s}^{2}(\bbZ^{2})$ have a positive eigenvalue. Moreover, the following assertions hold for all $\nu$ in the respective intervals $(0,\nu_{0}^{sg})$, $(0,\nu_{0}^{a})$, $(0,\nu_{0}^{v})$:
if $\mathbf q$ is of type $I_0$ then $\lambda > 0$ is an eigenvalue of $L_{B,\mathbf q}^{\alpha}$, respectively for the second grade, Navier-Stokes-$\alpha$ and Navier-Stokes-Voigt models, with eigenvector $(w_{n})$ satisfying Property \ref{eigvectorions}
if and only if $\lambda > 0$ is a solution to the following equations
\begin{equation}\label{eqiov-sg}
e_{0}(\lambda,\nu,\alpha)+f^{s}(\lambda,\nu,\alpha)+g^{s}(\lambda,\nu,\alpha)=0,
\end{equation}
\begin{equation}\label{eqiov-nsa}
l_{0}(\lambda,\nu,\alpha)+f^{a}(\lambda,\nu,\alpha)+g^{a}(\lambda,\nu,\alpha)=0,
\end{equation}
\begin{equation}\label{eqiov-nsv}
i_{0}(\lambda,\nu,\alpha)+f^{v}(\lambda,\nu,\alpha)+g^{v}(\lambda,\nu,\alpha)=0.
\end{equation}
\end{theorem}

\section{Future questions}
We mention the following questions that can be explored in the future. One question is to carry out numerical investigations to determine the effect of various parameters such as $\nu$ and $\alpha$ on the location of the unstable eigenvalue $\lambda$. A second direction of study is to use the Fredholm determinant characterization obtained in Section 3 to directly compute unstable eigenvalues, both theoretically and numerically. A third direction would be to try and understand the effect of the regularization parameter $\alpha$ on the nonlinear instability of the $\alpha$-models, see \cite{LV19} for related work on nonlinear stability results for the 2D Euler-$\alpha$ equations.




\end{document}